\documentclass{article}[12pt]
\usepackage[utf8]{inputenc}
\usepackage{amsmath,amssymb,amstext,amsthm, geometry, verbatim, tikz-cd, enumerate}
\newtheorem{theorem}{Theorem}[section]
\newtheorem{lemma}[theorem]{Lemma}
\newtheorem{proposition}[theorem]{Proposition}
\newtheorem{corollary}{Corollary}[theorem]
\theoremstyle{definition}
\newtheorem{definition}[theorem]{Definition}

\theoremstyle{remark}
\newtheorem{remark}[theorem]{Remark}

\newcommand \Z{\mathbb Z}

\newcommand \Q{\mathbb Q}
\newcommand \C{\mathbb C}

\newcommand \vp{\varphi}

\numberwithin{equation}{section}
\newcommand \mc{\mathcal}
\newcommand \mb{\mathbb}
\newcommand \mf{\mathfrak}
\DeclareMathOperator{\Hom}{Hom}
\DeclareMathOperator{\Pic}{Pic}

\DeclareMathOperator{\rank}{rank}
\DeclareMathOperator{\Span}{span}

\DeclareMathOperator{\Coh}{Coh}
\DeclareMathOperator{\ord}{ord}
\DeclareMathOperator{\Ext}{Ext}

\title{Moduli Spaces of Stably Irreducible Sheaves on Kodaira Surfaces}
\author{Eric Boulter\footnote{Department of Pure Mathematics, University of Waterloo, Waterloo, Canada.  Email: etboulte@uwaterloo.ca}}

\begin{document}
	
\maketitle

\begin{abstract}
 Moduli spaces of stably irreducible sheaves on Kodaira surfaces belong to the short list of examples of smooth and compact holomorphic symplectic manifolds, and it is not yet known how they fit into the classification of holomorphic symplectic manifolds by deformation type. This paper studies a natural Lagrangian fibration on these moduli spaces to determine that they are not K\"ahler or simply connected, ruling out most of the known deformation types of holomorphic symplectic manifolds.
\end{abstract}

\tableofcontents

\section{Introduction}
The study of holomorphic symplectic manifolds began with Bogomolov's classification of compact K\"ahler manifolds with trivial canonical bundle. These manifolds decompose up to finite \'etale cover as a product of a complex torus, irreducible Calabi-Yau manifolds, and irreducible holomorphic symplectic manifolds \cite{BeauSurvey,Bog}. It is generally very difficult to construct compact examples of holomorphic symplectic manifolds; nearly all constructions make use of the fact that the Hilbert scheme (or Douady space) of points over a holomorphic symplectic surface is holomorphic symplectic \cite{Beau}, as is a smooth and compact moduli space of stable sheaves with fixed Chern character on a hyperK\"ahler surface \cite{Mukai}.

By the Enriques-Kodaira classification, all compact holomorphic symplectic surfaces are complex tori, K3 surfaces, or primary Kodaira surfaces. Each of these holomorphic symplectic surfaces generates an infinite family of holomorphic symplectic manifolds via its Hilbert schemes (or Douady spaces) of points \cite{Beau}. These give rise to generalized Kummer varieties in the case of complex tori, and Bogomolov-Gaun manifolds in the case of primary Kodaira surfaces \cite{BogGuan,Guan}. For K3 surfaces and complex tori it has been shown that the moduli spaces of stable sheaves with fixed Chern character are deformation equivalent to the product of a Hilbert scheme of points with the Picard group of the surface \cite{O'GradyHodgeSt,Yoshioka} whenever they are smooth and compact. It is an open question whether this result also holds for primary Kodaira surfaces.

In the case of primary Kodaira surfaces, Toma showed that the moduli space of stable sheaves with fixed determinant and Chern character is holomorphic symplectic whenever it is smooth and compact \cite{TomaCpt}, and determined that a sufficient condition to guarantee smoothness and compactness of the moduli space is to take a Chern character corresponding to stably irreducible sheaves. Aprodu, Moraru, and Toma studied the two-dimensional moduli spaces of rank-2 stably irreducible sheaves over primary Kodaira surfaces, and determined that they are also primary Kodaira surfaces \cite{ApMorTom}. In higher dimensions, it is not yet known whether these moduli spaces are always deformation equivalent to Douady spaces of points over primary Kodaira surfaces.

In this paper, we determine that there are compact moduli spaces of stably irreducible sheaves on Kodaira surfaces of dimension $2n$ for every $n$. In addition, we show that these moduli spaces are non-K\"ahler and have no simply connected components. Douady spaces of points on Kodaira surfaces are the only other known examples of compact holomorphic symplectic manifolds with these properties. An interesting question is to determine whether these moduli spaces are deformation equivalent to Douady spaces of points on Kodaira surfaces or form a new class of examples. Towards answering this question, we analyse a natural fibration on these spaces, which is described in detail for dimensions 4 and 6 in section 5.

Consider a general compact complex surface $X$ with Gauduchon metric $g$, and consider the moduli space $\mc{M}_{r,\delta,c_2}^g(X)$ of $g$-stable coherent sheaves with rank $r$, determinant $\delta$, and second Chern class $c_2$ on X. In his paper \cite{TomaCpt}, Toma gives a sufficient condition for this moduli space to be smooth and compact: \begin{align*}&\text{Every }g\text{-semi-stable vector bundle } E \text{ with }\\ &\rank(E)=r,c_1(E)=c_1(\delta), c_2(E)\leq c_2&& (*)\\ &\text{ is }g\text{-stable.}\end{align*} 
When $X$ has odd first Betti number, this criterion is equivalent to requiring that every bundle $E$ with rank $r$, $c_1(E)=c_1(\delta)$ and $c_2(E)\leq c_2$ is irreducible. In this case, the compactification of the moduli space of stable bundles with rank $r$, determinant $\delta$, and second Chern class $c_2$ is isomorphic to the moduli space of stably irreducible torsion-free sheaves. Using Br\^inz\u anescu's sufficient conditions for a sheaf to be irreducible \cite{Brin}, we find a range of invariants for which ($*$) is satisfied when $X$ is a primary Kodaira surface. In particular, we show that the moduli spaces of rank-two sheaves which are smooth and compact can be of any even dimension.

In section 3, we review the theory of spectral curves on primary Kodaira surfaces, as discussed in \cite{BrMor,BrMor2,BrMoFM}, which classifies bundles over a Kodaira surface based on their restrictions to the fibres of the natural bundle map associated to the Kodaira surface. We then construct the space $\mb{P}_{\delta,c_2}$ of spectral curves for sheaves in $\mc{M}_{2,\delta,c_2}$. We also decompose $\mb{P}_{\delta,c_2}$ into a filtration based on the number of ``jumps" each spectral curve has. The spectral curves give a finer invariant for torsion-free sheaves than the Chern character and determinant, and in the rank-two case, the graph map $\mc{M}_{2,\delta,c_2} \to \mb{P}_{\delta,c_2}$ sending each sheaf to its spectral curve is an algebraic completely integrable system \cite{BrMor}. The fibres of this integrable system are the focus of sections 4 and 5. We show that $\mb{P}_{\delta,c_2}$ is a $\mb{P}^n$-bundle and give an explicit description of the $\mb{P}^n$-bundle structure for each choice of $\delta$ and $c_2$. In particular, we show that for $\Delta(2,\delta,c_2)\geq \frac{1}{2}$ $\mb{P}_{\delta,c_2}$ is never biholomorphic to the base $\mathrm{Sym}^n(B)$ of the natural Lagrangian fibration on any Douady space $X^{[n]}$ over a Kodaira surface. Finally, we review properties of elementary modifications of a rank-2 vector bundle, and use elementary modifications to prove that every irreducible component of a spectral curve in $\mb{P}_{\delta, c_2}$ is smooth when $(2,\delta,c_2)$ correspond to stably irreducible sheaves.

In section \ref{Fibres}, we review the construction of Br\^inz\u anescu and Moraru \cite{BrMoFM} of the fibres of the graph map above spectral curves without jumps, and describe the fibres of the graph map above spectral curves with exactly one jump. Since spectral curves with $k$ jumps can only occur when the moduli space has dimension at least $4k$, understanding these cases allows us to describe all of the fibres of the graph map when the dimension of the moduli space is less than 8. In order to look at the fibres above spectral curves with one jump, we use elementary modifications to parameterize the locally free sheaves, and the structure of the multiplicity one Quot scheme to parameterize the non-locally free sheaves. We also determine which non-locally free sheaves can occur as limit points of vector bundles in the same fibre.

In section \ref{Applications}, we use results from section \ref{Fibres} to prove the main result of the paper: 
\begin{theorem}
	Let $X$ be a primary Kodaira surface, and let $(\delta,c_2) \in \Pic(X)\times \Z$ be such that $\mc{M}_{2,\delta,c_2}$ has positive dimension and contains stably irreducible vector bundles. Then $\mc{M}_{2,\delta,c_2}(X)$ is a non-K\"ahler manifold with no simply connected components.
\end{theorem}
  In this section we also describe the fibration structure of moduli spaces with dimension at most 6 in more detail using the results from section \ref{Fibres}, as for these dimensions there are no spectral curves with more than one jump.
  
  The remainder of Section \ref{Applications} discusses comparisons between the moduli spaces $\mc{M}_{2,\delta,c_2}(X)$ and the Douady spaces $X^{[n]}$, as well as the graph map corresponding to moduli of stable rank-2 sheaves on an elliptically fibred abelian surface. Any moduli space of stable sheaves on an elliptically fibred abelian surface is birational to a Hilbert scheme of points, and the birational map can be constructed via allowable elementary modifications \cite[Chapter 8]{Fried}. A similar situation does not occur in the Kodaira surface case as the general bundle does not have allowable elementary modifications. We conclude with a discussion of possible avenues to reconcile this discrepancy, including an analysis of moduli spaces of vector bundles on a product of elliptic curves $C_1\times C_2$  where different choices of elliptic fibration structure give a description of the moduli space both in terms of a graph map and the birational map to $\Pic^0(C_1\times C_2)\times (C_1\times C_2)^{[n]}$.

\section{Preliminaries}\label{Prelim}
\subsection{Stably Irreducible Sheaves}
Let $X$ be a compact complex surface $X$ with Gauduchon metric $g$. Given invariants $r \in \Z^+, \delta \in \Pic(X)$, and $c_2 \in \Z$, we write $\mc{M}^g_{r,\delta,c_2}(X)$ as the moduli space of $g$-stable rank-$r$ torsion-free sheaves on $X$ with determinant $\delta$ and second Chern class $c_2$. In his paper \cite{TomaCpt}, Toma shows that a sufficient condition for the moduli space $\mc{M}^g_{r,\delta,c_2}(X)$ to be smooth and compact is 
\begin{align*}&\text{Every }g\text{-semi-stable vector bundle } E \text{ with }\\ &\rank(E)=r,c_1(E)=c_1(\delta), c_2(E)\leq c_2&& (*)\\ &\text{ is }g\text{-stable,}\end{align*} 
provided $\mc{M}^g_{r,\delta,c_2}(X)$ is non-empty.
\begin{definition}
	A coherent sheaf is \emph{reducible} if it contains a subsheaf of strictly lower positive rank. A sheaf $\mc{E} \in \Coh(X)$ is \emph{stably irreducible} if there is no reducible sheaf $\mc{F}\in \Coh(X)$ with $\text{ch}(\mc{E})=\text{ch}(\mc{F})$.
\end{definition}
In the case that the Betti number $b_1(X)$ is odd, the $(*)$ condition is equivalent to the condition that every vector bundle in $\mc{M}^g_{r,\delta,c_2}(X)$ is stably irreducible \cite{TomaCpt}. Since any sheaf which is not $g$-stable is automatically reducible, stably irreducible sheaves are stable for any choice of Gauduchon metric. Because of this, the choice of metric is irrelevant in these cases, so we write $\mc{M}_{r,\delta,c_2}(X)$ instead of $\mc{M}^g_{r,\delta,c_2}(X)$ when $b_1(X)$ is odd.

In order to determine conditions for the existence of stably irreducible bundles, we introduce the invariants $\Delta(r,c_1,c_2)$ and $t(r,c_1)$. The \emph{discriminant} $\Delta(r,c_1,c_2)$ is given by $$\Delta(r,c_1,c_2):=\frac{1}{r}\left(c_2-\frac{r-1}{2r}c_1^2\right).$$
For any sheaf $\mc{E}\in \Coh(X)$, we write $$\Delta(\mc{E}):=\Delta(\rank(\mc{E}), c_1(\mc{E}), c_2(\mc{E})).$$ If $X$ is not algebraic, then $\Delta(\mc{E})\geq 0$ \cite[Theorem 4.17]{Brin}. The $t$-invariant is defined as $$t(r,c_1):=-\frac{1}{2}\sup_{k\in \Z, 0<k<r} \left(\frac{1}{k(r-k)}\sup_{\alpha \in NS(X)}\left(\frac{kc_1}{r}-\alpha\right)^2\right)$$ \cite[Lemma 4.30]{Brin}. Since the intersection product is negative semi-definite on non-algebraic complex surfaces \cite[Corollary 2.9]{Brin}, $t(r,c_1)\geq 0$, and $t(r,c_1)=0$ whenever $c_1 \in m NS(X)$ for some positive integer $m$ with $\gcd(m,r)>1$. Again, we write $$t(\mc{E}):=t(\rank(\mc{E}), c_1(\mc{E})).$$

Any reducible sheaf $\mc{E}$ has $\Delta(\mc{E})\geq t(\mc{E})$ \cite[Lemma 4.30]{Brin}, so any $\mc{E}$ satisfying $$0\leq \Delta(\mc{E})<t(\mc{E})$$ is stably irreducible. When $X$ is a primary Kodaira surface, $\mc{M}_{r,\delta,c_2}(X)$ is non-empty of dimension $2r^2\Delta(r,\delta,c_2)$ for all $(r,\delta,c_2)$ with $\Delta(r,\delta,c_2)\geq 0$ \cite{TomaNote}, so in the range $$0\leq \Delta(r,\delta,c_2)<t(r,\delta),$$ $\mc{M}_{r,\delta,c_2}(X)$ is non-empty and must be a compact holomorphic symplectic manifold by Toma's condition \cite{TomaCpt}.

\begin{remark}\label{tensor}
	For any line bundle $\lambda \in \Pic(X)$, there is an isomorphism between $\mc{M}_{2, \delta, c_2}(X)$ and $\mc{M}_{2, \delta\otimes \lambda^{\otimes r}, c_2+a}(X)$ given by $\mc{E} \mapsto \mc{E}\otimes \lambda$,
	where $$a=(r-1)c_1(\lambda).c_1(\delta)+\frac{r^2-r}{2}c_1(\lambda)^2.$$ In the rank-2 case, we can use this fact along with the fact that the intersection product on a Kodaira surface is negative definite to find an isomorphic moduli space $\mc{M}_{2,\delta',c_2'}$ such that $c_1(\delta')^2=-8t(2,c_1(\delta'))$, since $$t(2,c_1)=-\frac{1}{8}\max\limits_{\alpha \in NS(X)}(c_1-2\alpha)^2.$$ Under this assumption, we have  $\Delta(2,c_1,c_2)<t(2,c_1)$ if and only if $c_2<0$. Because of this, every moduli space of rank-2 coherent sheaves in the stably irreducible range is isomorphic to a moduli space with invariants satisfying $c_1(\delta)^2=-8t(2,c_1(\delta))$ and $c_2<0$, so we can restrict to these cases when searching for examples.
\end{remark}
\subsection{Line Bundles on Kodaira Surfaces}
Given the data of a smooth genus one curve $B$, a positive integer $d$, a line bundle $\Theta \in \Pic^d(B)$, and a complex number $\tau$ with $|\tau|>1$, the quotient $X:=\Theta^*/\langle \tau\rangle$ via the standard $\C^*$ action on $\Theta^*$ gives a principal bundle with base $B$ and structure group $\C^*/\langle \tau\rangle$, otherwise known as a primary Kodaira surface. (Here $\Theta^*$ represents the complement of the zero section in the total space of $\Theta$.)
Let $\pi:X \to B$ be the induced projection map. We denote the fibre $\pi^{-1}(b)$ of any $b \in B$ by $T:=\C^*/\langle\tau\rangle$. Primary Kodaira surfaces have the following topological and analytic invariants \cite{barth2003compact}:
\begin{align*}
	H^1(X,\Z)=\Z^3, && H^2(X,\Z)=\Z^4\oplus \Z/d\Z, && H^3(X,\Z)=\Z^3\oplus \Z/d\Z,\\
	H^1(X,\mc{O}_X)=\C^2, && H^2(X,\mc{O}_X)=\C, && \Pic^0(X)=\pi^*(\Pic^0(B))\oplus \C^*. 
\end{align*}
Here the torsion part of $H^2(X,\Z)$ is generated by the fibre of $\pi$ \cite{Teleman}. Furthermore, every line bundle on $X$ with torsion first Chern class can be written as $\pi^*(H)\otimes L_\alpha$, where $H \in \Pic(B)$ and $L_\alpha$ is the bundle on $X$ with constant factor of automorphy $\alpha$, subject to the relation $\pi^*\Theta \otimes L_\tau=\mc{O}_X$. The bundles $L_\alpha$ occur as the quotient of $\Theta^*\times \C$ by the $\Z$-action with generator $(x,t)\mapsto (\tau x, \alpha t).$

The torsion-free component of the N\'eron--Severi group can be identified with the group \linebreak $\Hom(\Pic^0(B), \Pic^0(T))$ \cite{Brin}, and the intersection product on $NS(X)$ can be computed using this identification;  the self-intersection is given by $$\vp^2=-2\deg(\vp)$$ for any $\vp \in \Hom(\Pic^0(B), \Pic^0(T))$ \cite[Theorem 1.10, Remark 1.11]{Teleman}.

\begin{proposition}
	For any non-negative integer $n$ and positive integer $r$, there is a Kodaira surface $X$, a line bundle $\delta \in \Pic(X)$ and an integer $c_2$ so that $\mc{M}_{r,\delta,c_2}(X)$ is a compact holomorphic symplectic manifold of dimension $2n$.
\end{proposition}
\begin{proof}
	If $r=1$, then $\mc{M}_{1, \delta, n}(X)$ is isomorphic to the Douady space $X^{[n]}$ for any line bundle $\delta$ and Kodaira surface $X$, and $X^{[n]}$ is a compact holomorphic symplectic manifold of dimension $2n$ \cite{Beau}. This isomorphism can be written explicitly by mapping each length-$n$ set $Z \in X^{[n]}$ to $(\delta\otimes \mc{I}_Z)$, where $\mc{I}_Z$ is the ideal sheaf of holomorphic functions vanishing at $Z$.
	
	If $r>1$, choose $\alpha \in \mb{H}=\{z \in \C:\text{Im}(z)>0\}$ so that $\Span_\Q\{1,\alpha,\alpha^2\}$ has rank $3$, and let $X$ be any primary Kodaira surface with base $\C^*/\langle \exp(-2\pi i\alpha)\rangle$ and fibre $\C^*/\langle \exp(-2(rn-n+r)\pi i\alpha)\rangle$. Such a Kodaira surface will have $$NS(X)/\text{Tors}(NS(X))\cong\Hom(\Z+\alpha\Z, \Z+(rn-n+r)\alpha\Z)\cong \Z,$$ and any generator $a$ of the torsion-free component of $NS(X)$ will satisfy $a^2=-2(rn-n+r)$. If we consider the $t$-invariant in this case, we get \begin{align*}t(r,a)&=-\frac{a^2}{2}\min_{0<k<r} \frac{1}{k(r-k)} \min_{m \in \Z} \left(\frac{k}{r}-m \right)^2\\
		&=-\frac{a^2}{2}\min_{1\leq k\leq \frac{r}{2}} \frac{k}{(r-k)r^2}\\ &=-\frac{a^2}{2r^2(r-1)},\end{align*}
	using the fact that $\frac{1}{k(r-k)}\left(\frac{k}{r}-m\right)^2$ is invariant under $(k,m)\mapsto (r-k, 1-m)$. After setting $c_2=2n-rn-r+1$, $$\Delta(r,a,c_2)=\frac{1}{r}\left(2n-rn-r+1+\frac{r-1}{r}(rn-n+r)\right)=\frac{n}{r^2}$$
	and $$t(r,a)=-\frac{a^2}{2r^2(r-1)}=\frac{rn-n+r}{r^2(r-1)}=\frac{n}{r^2}+\frac{1}{r(r-1)}>\Delta(r,a,c_2),$$
	so for any line bundle $\delta$ with $c_1(\delta)=a$, the moduli space $\mc{M}_{r,\delta,c_2}$ will be a compact holomorphic symplectic manifold of dimension $2r^2\Delta(r,a,c_2)=2n$. 
\end{proof}

\subsection{The Douady Space of Points for a Kodaira Surface}\label{Douady}

The Douady space $M^{[n]}$ of a complex manifold $M$ is an analytic space parameterizing the coherent $\mc{O}_M$-modules with finite support of length $n$. In the case that $M$ has dimension 2, $M^{[n]}$ is itself a complex manifold, and a holomorphic symplectic structure on $M$ naturally lifts to $M^{[n]}$ when $M$ is 2-dimensional \cite{Beau}. In the following, we compile some known information about Lagrangian fibrations and holomorphic invariants of Douady spaces over Kodaira surfaces in order to compare them with the Lagrangian fibrations we will construct for moduli spaces of stably irreducible sheaves over Kodaira surfaces.

\subsubsection{Lagrangian Fibration Structure}
The following is closely adapted from the discussion in \cite{Lehn} of the Hilbert scheme of points on a K3 surface.

If $\pi:X \to B$ is a principal elliptic surface with fibre $T$, there is an induced abelian variety fibration on the Douady space $X^{[n]}$ given by the composition $\pi^{[n]}:=\varrho\circ \text{Sym}^n(\pi)$, where $\text{Sym}^n(\pi):\text{Sym}^n(X)\to \text{Sym}^n(B)$ is the induced map of symmetric products and $\varrho:X^{[n]}\to \text{Sym}^n(X)$ is the Douady-Barelet morphism. (As shown in \cite{CatCil}, $\text{Sym}^n(B)$ is isomorphic to the projectivization of an indecomposable bundle on $B$ of rank $n$ and degree $-1$.) We will focus on this fibration in the case of $n=2$, as in this case $\varrho$ is simply the blow-up of the diagonal $\Delta$ in $\text{Sym}^n(X)$. If $(b_1,b_2) \in \text{Sym}^2(B)\setminus \Delta_B$, the fibre $(\pi^{[2]})^{-1}(b_1,b_2)$ is  $T\times T$. For a point $(b,b) \in \Delta_B$, the fibre is given by the union of two irreducible components: $$\mb{P}(N_{\text{Sym}^2(X)/\Delta}|_{(\text{Sym}^2(\pi))^{-1}(b,b)})\cong \mb{P}(\mc{T}_X|_{T_b})$$ and $\text{Sym}^2(T)$. The symmetric product $\text{Sym}^2(T)$ can naturally be thought of as the set of effective divisors of degree 2, and it has a ruled surface structure given by sending a divisor to its linear equivalence class in $\Pic^2(T)\cong T$. The intersection of these components is given by the diagonal of $\text{Sym}^2(T)$ and the section $\mb{P}(T_{X/B}|_{T_b})$ in $\mb{P}(TX|_{T_b})$. Since $TX$ is isomorphic to $\pi^*\mc{A}$ with $\mc{A}$ the Atiyah bundle, $TX|_{T_b}$ is trivial and $T_{X/B}|_{T_b}$ is the mapping onto the second factor.

\subsubsection{Topology}
Using a result of de Cataldo and Migliorini \cite{CatMig} (due to G\"ottsche \cite{Gottsche} in the projective case), we have that the Betti numbers of $X^{[n]}$ satisfy \begin{equation}\label{gottsche}
    \sum\limits_{n\geq 0} p(X^{[n]},t)q^n=\prod\limits_{k\geq 0}\prod\limits_{j=0}^4\left(1-(-t)^{2k-2+j}q^k\right)^{(-1)^{j+1}b_j(X)}, 
\end{equation}
where $p(X,t)=\sum\limits_{j\geq 0}b_j(X)t^j$ is the Poincar\'e polynomial. (Note that truncating the product on the right at $k=n$ gives the correct coefficient for $q^i$ for each $i\leq n$, so the Betti numbers of $X^{[n]}$ for a particular choice of $n$ can be computed with this formula.) In particular, if $X$ is a Kodaira surface, the Betti numbers of $X^{[2]}$ are $(1,3,8,18,24,18,8,3,1)$ and the Betti numbers of $X^{[3]}$ are $(1,3,8,22,50,87,106,87,50,22,8,3,1).$ In addition, we have from \cite{Beau} that $\pi_1(X^{[n]})\cong H_1(X,\Z)=\Z^3\oplus \Z/d\Z$ for some $d \in \Z_{>0}$. We will refer to these results when discussing the fundamental group and Betti numbers of $\mc{M}_{2,\delta,c_2}(X)$ in section \ref{Applications}.

\section{Spectral Curves}\label{Spectral}
If $\mc{E} \in \Coh(X)$ is torsion-free, we can associate to $\mc{E}$ the sheaf $\mc{L}_{\mc{E}}=R^1{p_2}_*(\mc{P}\otimes {p_1}^*\mc{E})$ on $B\times \C^*\subseteq B\times \Pic^0(X)$, where $\mc{P}$ is the universal line bundle on $X\times_B (B\times \C^*)\subseteq X\times_B(B\times \Pic^0(X))$, and the $p_i$'s are the morphisms corresponding to the fibred product. The resulting $\mc{L}_{\mc{E}}$ is a torsion sheaf supported on an effective divisor $\tilde{S}_{\mc{E}}$ in $X\times \C^*$, consisting of points $(b,a)$ with multiplicity $h_1(T, (\mc{E}\otimes L_a)|_{\pi^{-1}(b)})$. Since $L_\tau \in \pi^*(\Pic(B))$ and all bundles on $B$ pull back to bundles which are trivial on all fibres of $\pi$, $\tilde{S}_{\mc{E}}$ descends to a divisor $S_{\mc{E}}$ on $J(X):=B\times T^*$, where $T^*:=\Pic^0(T)$. We call $S_{\mc{E}}$ the \emph{spectral curve} of $\mc{E}$. We can describe the spectral curve more concretely as $$S_{\mc{E}}=\{(b,\lambda) \in B\times T^*: H^1(T, \mc{E}|_{\pi^{-1}(b)}\otimes \lambda)\neq 0\},$$ with the multiplicity of $(b,\lambda)$ given by $h^1(T, \mc{E}|_{\pi^{-1}(b)}\otimes \lambda)$. Since the morphism $$\iota^*:H^2(X,\Z)\to H^2(T,\Z)$$ corresponding to the inclusion is zero \cite{Teleman}, $\mc{E}_{\pi^{-1}(b)}\otimes \lambda$ is a degree zero bundle for all $(b,\lambda) \in B\times T^*$. This fact together with the classification of vector bundles on genus 1 curves gives that $S_E\cap \{b\}\times \Pic^0(T)$ contains at most $\rank(\mc{E})$ points if and only if $\mc{E}_{\pi^{-1}(b)}$ is semistable. Thus the spectral curve has the form $$S_{\mc{E}}=\overline{C}+\sum\limits_{b \in U_{\mc{E}}}\mu_b(\{b\}\times T^*)$$
for some integers $\mu_b$, where $\overline{C}$ is an $r$-section of $J(X)\to B$ and $U_{\mc{E}}:=\{b \in B: \mc{E}|_{\pi^{-1}(b)} \text{ is unstable}\}.$
\begin{definition}
	We say that a vector bundle $E$ has a \emph{jump} at $b$ if $E|_{\pi^{-1}(b)}$ is unstable, and the \emph{multiplicity} $\mu(E,b)$ of the jump is the multiplicity $\mu_b$ of $b\times T^*$ in $S_\mc{E}$.
\end{definition}
\subsection{The Graph of a Rank-2 Sheaf}

In the case of rank-2 sheaves, for each $\delta \in \Pic(X)$ we can define an involution $$\iota_\delta: (b,\lambda) \mapsto (b,\delta|_{\pi^{-1}(b)}\otimes \lambda^{-1})$$ on $J(X)$. (This involution only depends on the class of $\delta$ in $\Pic(X)/\pi^*(\Pic(B))$.) The rank-2 sheaves $\mc{E}$ with $\det{\mc{E}}=\delta\otimes \pi^*(\lambda)$ for some $\lambda \in \Pic(B)$ are precisely the sheaves whose spectral curves are invariant under the action of $\iota_\delta$. Thus these spectral curves descend to the ruled surface $\mb{F}_\delta:=J(X)/\iota_\delta$ with induced projection $\rho:\mb{F}_\delta\to B$. By \cite{BrMor}, this ruled surface can be described as $\mb{F}_\delta=\mb{P}(V_\delta)$, where $$V_\delta:={q_1}_*\left(\mc{O}_{J(X)}\left(\mc{S}_{\mc{O}_X}+\mc{S}_\delta\right)\right),$$
and $q_1:J(X)\cong B\times T^*\to B$ is the projection map. The bundle $V_\delta$ is a rank-2 semi-stable vector bundle on $B$ of degree $-\frac{c_1(\delta)^2}{2}$\cite[Lemma 3.8]{BrMor}.
\begin{remark}
	A rank-2 torsion-free sheaf $\mc{E}$ on $X$ is irreducible if and only if $S_\mc{E}=\overline{C}+\sum\limits_{i=1}^k \{b_i\}\times T^*$ with $\overline{C}$ irreducible \cite{BrMor}.
\end{remark}
\begin{definition}
	For any rank-2 torsion-free sheaf $\mc{E}$ with $\det(\mc{E})=\delta$, the \emph{graph} of $\mc{E}$ is the set $S_\mc{E}/\iota_\delta \subset \mb{F}_\delta$.
\end{definition}
\begin{proposition}[Br\^inz\u anescu--Moraru \cite{BrMor}]
	Let $\mc{E}$ be a rank-2 torsion-free sheaf with determinant $\delta$ and second Chern class $c_2$, and let $G$ be the graph of $\mc{E}$. Then $G$ is an effective divisor linearly equivalent to $A_\delta+\rho^*\mf{b}$, where $A_\delta$ is the graph of $\mc{O}_X\oplus \delta$, $\rho:\mb{F}_\delta\to B$ is the induced projection map, and $\mf{b} \in \Pic^{c_2}(B)$.
\end{proposition}

\begin{definition}
	Given a line bundle $\delta \in \Pic(X)$ and an integer $c_2$, the \emph{space of graphs} $\mb{P}_{\delta,c_2}$ is the set of all divisors in $\mb{F}_\delta$ linearly equivalent to $A_\delta +\rho^*\mf{b}$ for some $\mf{b} \in \Pic^{c_2}(B)$. The \emph{graph map} $\mc{G}:\mc{M}_{2,\delta,c_2}(X)\to \mb{P}_{\delta,c_2}$ is the holomorphic map taking each sheaf $\mc{E} \in \mc{M}_{2,\delta,c_2}(X)$ to its graph. (Equivalently, $\mc{G}$ is the map sending each sheaf to its spectral curve. We use these two definitions interchangeably.)
\end{definition}
In the stably irreducible case, we have the following result.
\begin{proposition}[Br\^inz\u anescu--Moraru \cite{BrMor2}]\label{surjective}
	The graph map $\mc{G}:\mc{M}_{2,\delta,c_2}\to \mb{P}_{\delta,c_2}$ is surjective whenever $t(2,c_1(\delta))>0$ and $c_2<0$.
\end{proposition}
Since in the stably irreducible case the graph map is a natural fibration of $\mc{M}_{2,\delta,c_2}(X)$ over $\mb{P}_{\delta,c_2}$, understanding the base and fibres of the fibration will allow us to determine topological properties of $\mc{M}_{2,\delta,c_2}$ in section 5. In the remainder of this section, we investigate the structure of the $\mb{P}_{\delta,c_2}$ by analyzing the divisors of the form $A_\delta+\rho^*\mf{b}$ for $\mf{b} \in \Pic^{c_2}(B)$.

\begin{proposition}\label{Pnbundle}
    In the case that $\Delta(2,\delta,c_2)>0$, $\mb{P}_{\delta,c_2}\cong\mb{P}(E_{\delta,c_2})$, where $$E_{\delta,c_2}:=(\pi_1)_*\left(\pi_{12}^*\mc{P}_{c_2}(b_0)\otimes \pi_{23}^*\left(\mc{O}_{J(X)}\left(\mc{S}_0+\mc{S}_\delta\right)\right)\right),$$
    $\mc{P}_{c_2}(b_0)$ is the Poincar\'e bundle $\mc{P}_{c_2}=\mc{O}_{B\times B}(\Delta+(c_2-1)B\times\{b_0\}-\{b_0\}\times B)$ of degree $c_2$ line bundles on $B$ with base point $b_0\in B$, and the projections $\pi_{ij}, p_k,q_\ell$ are as in the commutative diagram below.
    $$\begin{tikzcd}
                                             & B \times B \times T^* \arrow[rd, "\pi_{23}"] \arrow[ldd, "\pi_1"', bend right=80, shift right=4] \arrow[ld, "\pi_{12}"'] &                                          \\
B\times B \arrow[d, "p_1"] \arrow[rd, "p_2"] &                                                                                                                          & J(X) \arrow[ld, "q_1"'] \arrow[d, "q_2"] \\
B                                            & B                                                                                                                        & T^*                                     
\end{tikzcd}$$
\end{proposition}
\begin{proof}
    Since $\pi_1=p_1\circ \pi_{12}$, we have $E_{\delta,c_2}=(p_1)_*\left(\mc{P}_{c_2}(b_0)\otimes (\pi_{12})_*\left(\pi_{23}^*\mc{O}_{J(X)}(\mc{S}_0+\mc{S}_\delta)\right)\right)$ by the projection formula. Since $q_1\circ\pi_{23}=p_2\circ\pi_{12}$, we can apply the base change theorem to obtain \begin{align*}E_{\delta,c_2}&\cong (p_1)_*\left(\mc{P}_{c_2}(b_0)\otimes p_2^*\left((q_1)_*(\mc{O}_{J(X)}(\mc{S}_0+\mc{S}_\delta)\right)\right)\\
    &=(p_1)_*\left(\mc{P}_{c_2}(b_0)\otimes p_2^*V_\delta\right).\end{align*}
    For any $b \in B$, $\left(\mc{P}_{c_2}(b_0)\otimes p_2^*V_\delta\right)|_{\{b\}\times B}\cong \mc{O}_B(b+(r-1)b_0)\otimes V_\delta$, so since $\Delta(2,\delta,c_2)>0$, $\lambda_b\otimes V_\delta$ has positive degree, $E_{\delta,c_2}$ is locally free and the fibre above $b$ in $E_{\delta,c_2}$ is indeed $$H^0(B, \mc{O}_{B}(b+(r-1)b_0)\otimes V_\delta).$$ From this we conclude that $\mb{P}(E_{\delta,c_2})\cong \mb{P}_{\delta,c_2}$.
\end{proof}
\begin{lemma}\label{no-sections}
    For any choice of $\delta \in \Pic(X)$, $c_2\in \Z$, the bundle $E_{\delta,c_2}$ has no global sections.
\end{lemma}
\begin{proof}
    Note that \begin{align*}H^0(B,E_{\delta,c_2})&=H^0(B\times B \times T^*, \pi_{12}^*\mc{P}_{c_2}(b_0)\otimes \pi_{23}^* \mc{O}_{J(X)}(\mc{S}_0+\mc{S}_\delta))\\
    &=H^0(J(X), (\pi_{23})_*(\pi_{12}^*\mc{P}_{c_2}(b_0)\otimes \pi_{23}^*\mc{O}_{J(X)}(\mc{S}_0+\mc{S}_\delta))).\end{align*} Using the projection formula, we have $H^0(B, E_{\delta,c_2})=H^0(J(X), (\pi_{23})_*\pi_{12}^*\mc{P}_{c_2}(b_0)\otimes \mc{O}_{J(X)}(\mc{S}_0+\mc{S}_\delta)).$ From the base change theorem, $(\pi_{23})_*\pi_{12}^*\mc{P}_{c_2}(b_0)\cong q_1^*(p_2)_*\mc{P}_{c_2}.$ Since the restriction to any fibre of $p_2$ is a degree-0 line bundle which is trivial only for $p_2^{-1}(b_0)$, $(p_2)_*\mc{P}_{c_2}=0$. From this we conclude that $H^0(B, E_{\delta,c_2})=0$.
\end{proof}

For the next proof, we will employ the notation of \cite{CatCil} of $E_r(b_0) := (p_1)_*(\mc{P}_{r}(b_0))$, which is a stable bundle of degree $-1$ and rank $r$ for all $r>0$.
\begin{proposition}\label{spectral} Let $\delta\in \Pic(X)$ and $c_2\in \Z$ be such that $\Delta(2,\delta,c_2)\geq 0$.
	\begin{enumerate}
		\item[i.] If $\Delta(2,\delta,c_2)=0$, then $\mb{P}_{\delta,c_2}$ consists of $2$ points.
		\item[ii.] If $\Delta(2,\delta,c_2)>0$ and $4\Delta(2,\delta,c_2)$ is odd, then $\mb{P}_{\delta,c_2}\cong \mb{P}(E_{\delta,c_2})$, where $E_{\delta,c_2}$ is a stable bundle of rank $4\Delta(2,\delta,c_2)$ and degree -2.
		\item[iii.] If $\Delta(2,\delta,c_2)>0$ and $4\Delta(2,\delta,c_2)$ is even, then $\mb{P}_{\delta,c_2}\cong \mb{P}(E_{\delta,c_2})$, where $E_{\delta,c_2}$ is the direct sum of two stable bundles, each of rank $2\Delta(2,\delta,c_2)$ and degree $-1$.
	\end{enumerate}
\end{proposition}
\begin{proof}
    We begin by constructing a long exact sequence involving the $E_{\delta,c_2}$ which will be helpful in later computations. Note first that there is a natural exact sequence \begin{equation}\label{degrees}\begin{tikzcd}0 \arrow[r] & \mc{P}_{c_2-1}(b_0) \arrow[r] & \mc{P}_{c_2}(b_0) \arrow[r] & \mc{O}_{B \times \{b_0\}} \arrow[r] & 0
    \end{tikzcd}\end{equation}
    relating Poincar\'e bundles of adjacent degrees \cite{CatCil}. Pulling back by $\pi_{12}$ and twisting by the line bundle $\mc{L}_\delta:=\pi_{23}^*\mc{O}_{J(X)}(\mc{S}_0+\mc{S}_\delta)$ gives a new short exact sequence $$\begin{tikzcd}0 \arrow[r] & \pi_{12}^*\mc{P}_{c_2-1}\otimes \mc{L}_\delta \arrow[r] & \pi_{12}^*\mc{P}_{c_2} \otimes \mc{L}_\delta \arrow[r] & \mc{O}_{B\times \{b_0\}\times T^*}\otimes \mc{L}_\delta \arrow[r] & 0.
    \end{tikzcd}$$
    If we pushforward by $\pi_{12}$, we get 
    $$\begin{tikzcd}0 \arrow[r] & \mc{P}_{c_2-1}(b_0)\otimes (\pi_{12})_*\mc{L}_\delta \arrow[r] & \mc{P}_{c_2}(b_0)\otimes (\pi_{12})_*\mc{L}_\delta \arrow[ld] \\
    & (\pi_{12})_*(\mc{O}_{B\times\{b_0\}\times T^*}\otimes \mc{L}_\delta) \arrow[r]   & \mc{P}_{c_2-1}(b_0)\otimes R^1(\pi_{12})_*\mc{L}_\delta
    \end{tikzcd}$$
    from the projection formula. By the base change theorem, $$R^1(\pi_{12})_*\mc{L}_\delta\cong p_2^*R^1(q_1)_*\mc{O}_{J(X)}(\mc{S}_0+\mc{S}_\delta)=0,$$
    so the previous exact sequence is short exact.
    Finally, we can pushforward by $p_1$ to obtain the long exact sequence
     $$\begin{tikzcd}0 \arrow[r] & (p_1)_*\left(\mc{P}_{c_2-1}(b_0)\otimes p_2^*V_\delta\right)\arrow[r] & E_{\delta,c_2} \arrow[r] & (\pi_1)_*\left(\mc{O}_{B\times\{b_0\}\times T^*}\otimes \mc{L}_\delta\right) \arrow[lld] \\
   & R^1(p_1)_*(\mc{P}_{c_2-1}(b_0)\otimes p_2^*V_\delta) \arrow[r] & R^1(p_1)_*(\mc{P}_{c_2}(b_0)\otimes p_2^*V_\delta) \arrow[r] & \ldots .
    \end{tikzcd}$$
    Note that $\left(\mc{P}_{c_2}\otimes V_\delta\right)|_{\{b\}\times B}$ is semistable of positive degree when $\Delta(2,\delta,c_2)>0$, and any rank-2 semistable vector bundle of positive degree on a genus one curve has trivial first cohomology, so we have $R^1(p_1)_*\left(\mc{P}_{c_2}\otimes p_2^*V_\delta\right)=0$. We also have $$(\pi_1)_*\left(\mc{O}_{B\times\{b_0\}\times T^*}\otimes \mc{L}_\delta\right) \cong (q_1)_*(\mc{O}_{B\times T^*}(B\times \{0\}+B\times \{\delta_{b_0}\}))=(q_1)_*q_2^*\mc{O}_{T^*}(0+\delta_{b_0}),$$ so $$(\pi_1)_*\left(\mc{O}_{B\times\{b_0\}\times T^*}\otimes \mc{L}_\delta\right)) \cong \mc{O}_B\oplus\mc{O}_B.$$
    
    This gives the exact sequence \begin{equation}\label{PB}
        \begin{tikzcd}
            0 \arrow[r] & (p_1)_*(\mc{P}_{c_2-1}(b_0)\otimes p_2^*V_\delta) \arrow[r] & E_{\delta,c_2} \arrow[lld] \\
             \mc{O}_B\oplus \mc{O}_B \arrow[r] & R^1(p_1)_*(\mc{P}_{c_2-1}(b_0)\otimes p_2^*V_\delta) \arrow[r] & 0.
        \end{tikzcd}
    \end{equation}
    
    We now consider the case of $\Delta(2,\delta,c_2)=0$. The exact sequence \eqref{PB} for $E_{\delta,c_2+1}$ is then
    $$\begin{tikzcd} 0 \arrow[r] & E_{\delta,c_2+1} \arrow[r] & \mc{O}_B\oplus \mc{O}_B \arrow[r] & R^1(p_1)_*(\mc{P}_{c_2}(b_0)\otimes p_2^*V_\delta) \arrow[r] & 0,
    \end{tikzcd}$$
    where $R^1(p_1)_*(\mc{P}_{c_2}(b_0)\otimes p_2^*V_\delta)$ is a torsion sheaf whose stalk at $b \in B$ is given by $H^0(B, \mc{O}_B(b+(c_2-1)b_0)\otimes V_\delta)$. Suppose for a contradiction that $V_\delta$ is an indecomposable bundle. Then $R^1(p_1)_*(\mc{P}_{c_2}(b_0)\otimes p_2^*V_\delta)=\C_b$ for some $b \in B$, and by \cite{Boozer} we have $E_{\delta,c_2+1}\cong \mc{O}_B\oplus \mc{O}_B(-b)$, which has a non-zero section. This contradicts Lemma \ref{no-sections}, so $V_\delta$ is decomposable. Thus $V_\delta\cong \lambda\otimes (L\oplus L^{-1})$ for some $\lambda \in \Pic^{-c_2}(B)$ and $L\in \Pic^0(B)$.
    
    From the fact that $\mc{M}_{2,\delta,c_2}(X)$ is finite when $\Delta(2,\delta,c_2)=0$, there must be a line bundle $\lambda' \in \Pic^{c_2}(B)$ such that $H^0(B,\lambda'\otimes V_\delta)=\C.$
    This implies that $L\not\cong L^{-1}$, so $\mb{P}_{\delta,c_2}$ consists of two points corresponding to $\lambda'=\lambda^{-1}\otimes L$ and $\lambda'=\lambda^{-1}\otimes L^{-1}$.

    In the case where $\Delta(2,\delta,c_2)>0$ and $4\Delta(2,\delta,c_2)$ is odd, the fact that $V_\delta$ is semistable implies that $V_\delta\cong \lambda\otimes \mc{F}_p$, where $\lambda \in \Pic^{\frac{-c_1^2-2}{4}}(B)$, and $\mc{F}_p$ is the unique non-trivial extension of $\mc{O}_{B}(p)$ by $\mc{O}_B$ with $p \in B$. Thus $\mc{P}_{c_2}(b_0)\otimes p_2^*V_\delta$ fits into the exact sequence $$\begin{tikzcd} 0 \arrow[r] & \mc{P}_{c_2}(b_0)\otimes p_2^*\lambda \arrow[r] & \mc{P}_{c_2}(b_0)\otimes p_2^*V_\delta \arrow[r] & \mc{P}_{c_2}(b_0)\otimes p_2^*(\lambda(p)) \arrow[r] & 0.
    \end{tikzcd}$$
    Set $r=2\Delta(2,\delta,c_2)-\frac{1}{2}.$ Then there exist points $b_1,b_2\in B$ and line bundles $\lambda_1, \lambda_2 \in \Pic^0(B)$ such that $\mc{P}_{c_2}(b_0)\otimes p_2^*\lambda=\mc{P}_r(b_1)\otimes p_1^*\lambda_1$ and $\mc{P}_{c_2}(b_0)\otimes p_2^*(\lambda(p))=\mc{P}_{r+1}(b_2)\otimes p_1^*\lambda_2$.
    Pushing forward by $p_1$ gives the exact sequence $$\begin{tikzcd}0 \arrow[r] & (p_1)_*(\mc{P}_{r}(b_1))\otimes \lambda_1 \arrow[r] & E_{\delta,c_2}\otimes \lambda_2 \arrow[r] & E_{r+1}(b_2)\otimes \lambda_2 \arrow[lld]\\
    & R^1(p_1)_*(\mc{P}_{r}(b_1))\otimes \lambda_1 \arrow[r] & 0,
    \end{tikzcd}$$
    since we showed previously that $R^1(p_1)_*(\mc{P}_{c_2}(b_0)\otimes p_2^*V_\delta)=0$.
    
    If $\Delta(2,\delta,c_2)=\frac{1}{4}$ and therefore $r=0$, $(p_1)_*(\mc{P}_r(b_1))=0$ and $R^1(p_1)_*(\mc{P}_r(b_1))=\C_{b_1}$. Therefore when $\Delta(2,\delta,c_2)=\frac{1}{4}$, we have $E_{\delta,c_2}\cong E_1(b_2)\otimes \mc{O}_B(-b_1)\otimes \lambda_2\cong \mc{O}_B(-b_1-b_2)\otimes \lambda_2$ \cite{CatCil}. Thus $E_{\delta,c_2}$ is some line bundle of degree -2 on $B$.
    
    If instead $\Delta(2,\delta,c_2)>\frac{1}{4}$, $(p_1)_*(\mc{P}_r(b_1))=E_r(b_1)$ and $R^1(p_1)_*(\mc{P}_r(b_1))=0$, so $E_{\delta,c_2}$ is an extension of $E_{r+1}(b_2)\otimes \lambda_2\cong E_{r+1}(b'_2)$ by $E_r(b_1)\otimes \lambda_1\cong E_r(b'_1)$.
    
    The extension class corresponding to $E_{\delta,c_2}$ is an element of $$\Ext^1(E_{r+1}(b'_2), E_r(b'_1))=H^1(B,E_{r+1}(b'_2)^\vee \otimes E_r(b'_1)).$$ The tensor product of a stable bundle of rank $r+1$ with a stable bundle of rank $r$ is another stable bundle of rank $r(r+1)$ \cite[Lemma 28]{Atiyah}, so in particular there is some point $b'_3$ such that $E_{r+1}(b'_2)^\vee \otimes E_r(b'_1)=E_{r(r+1)}(b'_3)$. As shown in \cite{CatCil}, $E_{r(r+1}(b'_3)$ is an extension $$\begin{tikzcd}0 \arrow[r] & E_{r(r+1)}(b'_3)\arrow[r] & E_{r(r+1)+1}(b'_3)\arrow[r] & \mc{O}_B \arrow[r] & 0,
    \end{tikzcd}$$
   and the corresponding long exact sequence in cohomology induces an isomorphism $$H^0(B,\mc{O}_B)\cong H^1(B,E_{r(r+1)}(b'_3)).$$
    
    Similarly, the extension class corresponding to $\mc{P}_{c_2}(b_0)\otimes p_2^*V_\delta$ is a non-zero element of $$\Ext^1(\mc{P}_{c_2}(b_0)\otimes p_2^*\lambda(p), \mc{P}_{c_2}(b_0)\otimes p_2^*\lambda)=H^1(B\times B, p_2^*(\mc{O}_B(-p))).$$
    Since $p_1$ maps a surface to a curve, its Leray spectral sequence degenerates at page 2, giving $$H^1(B\times B, p_2^*(\mc{O}_B(-p)))=H^1(B, (p_1)_*(p_2^*(\mc{O}_B(-p)))\oplus H^0(B,R^1(p_1)_*(p_2^*(\mc{O}_B(-p))).$$
    The base change theorem gives that $R^i(p_1)_*(p_2^*(\mc{O}_B(-p)))$ is the trivial bundle on $B$ of rank $h^i(B, \mc{O}_B(-p))$, so $H^1(B\times B, p_2^*(\mc{O}_B(-p)))=H^0(B, \mc{O}_B).$ We now have an isomorphism mapping the extension class of $\mc{P}_{c_2}(b_0)\otimes p_2^*V_\delta$ to the extension class of $E_{\delta,c_2}$ via pushforward by $p_1$, so $E_{\delta,c_2}$ is the unique non-split extension of $E_{r+1}(b'_2)$ by $E_{r+1}(b'_1)$. We now show that this extension is a stable bundle. Suppose for a contradiction that $\mc{F}$ is a non-trivial sub-bundle of $E_{\delta,c_2}$ with $\mu(\mc{F})\geq \mu(E_{\delta,c_2})$. Without loss of generality, we can assume that $\mc{F}$ is stable. If $\mc{F}$ has degree $-2$, then $\mu(\mc{F})<\mu(E_{\delta,c_2})$, so $\mc{F}$ can only be a destabilizing bundle of $E_{\delta,c_2}$ if $\deg(\mc{F})\geq -1$. No subbundle of $E_r(b'_1)$ is destabilizing, so there must be a non-zero morphism $f:\mc{F} \to E_{r+1}(b'_2)$ given by the inclusion of $\mc{F}$ into $E_{\delta,c_2}$ followed by projection to $E_{r+1}(b'_2)$. If $f$ is surjective, $E_{r+1}(b'_2)$ is a quotient of $\mc{F}$ and $\mu(\mc{F})\leq \mu(E_{r+1}(b'_2)$. Then $\mc{F}$ must be of the form $E_k(q)$ for some $0<k\leq r+1$ and $q \in B$. Since $f$ is surjective, this is only possible if $\mc{F}=E_{r+1}(b'_2)$. Since the only endomorphisms of $E_{r+1}(b'_2)$ are constant multiples of the identity, this would induce a splitting of $E_{\delta,c_2}$, so $f$ cannot be surjective. Since $f$ is not surjective, then $\mu(\mc{F})\leq \mu(\mathrm{im}(f))<\mu(E_{r+1}(b'_2))$ by indecomposability of $\mc{F}$ and stability of $\mc{E}_{r+1}(b'_2)$. We now have $\deg(\mc{F})\geq -1$ and $\mu(\mc{F})< \frac{-1}{r+1}$, so $\mu(\mc{F})\leq \frac{-1}{r}<\mu(E_{\delta,c_2})$. This implies that $\mc{F}$ cannot be a destabilizing bundle for $E_{\delta,c_2}$, so $E_{\delta,c_2}$ is stable.
    
    Finally, we consider the case of $\Delta(2,\delta,c_2)>0$ with $4\Delta(2,\delta,c_2)$ even. As shown in the case where $\Delta(2,\delta,c_2)=0$, $V_\delta=\lambda\otimes (L\oplus L^{-1})$ where $\lambda \in \Pic^{-\frac{c_1^2(\delta)}{4}}(B)$ and $L \in \Pic^0(B)$. From this we conclude that $E_{\delta,c_2}\cong (p_1)_*(\mc{P}_{c_2}(b_0)\otimes \lambda\otimes L)\oplus (p_1)_*(\mc{P}_{c_2}(b_0)\otimes \lambda\otimes L^{-1}).$ Set $r=2\Delta(2,\delta,c_2)$. Then there are $b_1,b_2\in B$ and $\lambda_1,\lambda_2 \in \Pic^0(B)$ such that $\mc{P}_{c_2}(b_0)\otimes \lambda\otimes L\cong\mc{P}_r(b_1)\otimes p_1^*\lambda_1$ and $\mc{P}_{c_2}(b_0)\otimes \lambda\otimes L^{-1}\cong \mc{P}_r(b_2)p_1^*\otimes \lambda_2$, so $E_{\delta,c_2}\cong E_r(b'_1)\oplus E_r(b'_2)$ for some $b'_1,b'_2 \in B$.
\end{proof}
\begin{remark}
    Recall that the natural Lagrangian fibration on the Douady space $X^{[n]}$ has base $\text{Sym}^n(B)\cong \mb{P}(E_n(u))$ for some $u \in B$ \cite{CatCil}, so in particular $\mb{P}_{\delta,c_2}$ can only be isomorphic to $\text{Sym}^n(B)$ when $\Delta(2,\delta,c_2)=\frac{1}{4}$. 
\end{remark}

\begin{remark}\label{filter}
	We can write the space of graphs $\mb{P}_{\delta,c_2}$ as the union $$\mb{P}_{\delta,c_2}=\bigcup\limits_{k=0}^{\lfloor 2\Delta(2,\delta,c_2)\rfloor} \mb{P}_{\delta,c_2}^k,$$ where $$\mb{P}_{\delta,c_2}^k:=\{S \in \mb{P}_{\delta,c_2}: S=C+\sum\limits_{i=1}^k \{b_i\}\times T^*, C \text{ a bisection of } J(X)\}.$$
	Note that $\bigcup\limits_{i\geq k}\mb{P}_{\delta,c_2}^i$ is a closed subset of codimension $k$ in $\mb{P}_{\delta,c_2}$, and since for any spectral curve $C+\sum\limits_{i=1}^k \{b_i\}\times T^* \in \mb{P}_{\delta,c_2}^k$ we have $C \in \mb{P}_{\delta,c_2-k}^0$, there is an isomorphism $$\mb{P}_{\delta,c_2}^k \cong \mb{P}_{\delta,c_2-k}^0\times \text{Hilb}^k(B)$$
	for every integer $0\leq k \leq 2\Delta(2,\delta,c_2)$.
\end{remark}

\subsection{Spectral Curves and Elementary Modifications}
In order to understand vector bundles with jumps, the main method is to study their \emph{elementary modifications}. Given a rank-2 vector bundle $E$ on a complex manifold $X$, a smooth effective divisor $D$, a line bundle $\lambda$ on $D$, and a surjective sheaf map $g:E|_D\to \lambda$, the elementary modification $E'$ of $E$ by $(D,\lambda)$ is the unique vector bundle satisfying the exact sequence $$\begin{tikzcd}
	0 \arrow[r] & E' \arrow[r] & E\arrow[r] & \iota_*\lambda\arrow[r] & 0,
\end{tikzcd}$$
where $\iota:D\to X$ is the inclusion map. The invariants of an elementary modification are given by \begin{align*}\det(E')=\det(E)\otimes \mc{O}_X(-D), && c_2(E')=c_2(E)+c_1(E).[D]+\iota_*c_1(\lambda).\end{align*}
In the case that $X$ is a Kodaira surface and $D$ is a prime divisor, the divisor $D$ is of the form $\pi^{-1}(b)$ for some $b \in B$. Since $D$ has torsion first Chern class, in this case the determinant and second Chern class are related by \begin{align*}
	\det(E')=\det(E)\otimes \pi^*(\mc{O}_B(-b)), && c_2(E')=c_2(E)+\deg(\lambda).
\end{align*}
If a vector bundle $E$ has a jump at $b$, there is a unique elementary modification of $E$ along $\pi^{-1}(b)$ by a negative degree bundle, called the \emph{allowable elementary modification} of $E$ at $b$ \cite[Section 4.1.2]{MorHopf}; particularly, since $E|_{\pi^{-1}(b)}$ is unstable, it is of the form $\lambda \oplus (\lambda^*\otimes \det(E)|_{\pi^{-1}(b)})$ for some $\lambda \in \Pic^{-h}(T^*)$ with $h>0$, with the map $g:E|_{\pi^{-1}(b)}\to \lambda$ given by projection onto the first coordinate. The elementary modification $E'$ then has $\mu(E',b)=\mu(E,b)+\deg(\lambda).$

\begin{proposition}[Br\^inz\u anescu--Moraru \cite{BrMor}]\label{FiniteUnstable}
	If $\mc{E}$ is a rank-$2$ torsion-free sheaf on $X$, then $\mc{E}$ has finitely many jumps, and $$\sum\limits_{b\in U_\mc{E}} \mu(\mc{E}, b)\leq 2\Delta(\mc{E}).$$
\end{proposition}
\begin{proof}
	First, suppose that $\mc{E}$ is a vector bundle. If $\mc{E}|_{\pi^{-1}(b)}$ is unstable, it must split by the classification of rank-2 bundles on elliptic curves with $\mc{E}|_{\pi^{-1}(b)}=L\oplus (\det(\mc{E}|_{\pi^{-1}(b)})\otimes L^{-1}$ for some $L \in \Pic^{-k}(T)$ with $k>0$, and the allowable elementary modification $$\begin{tikzcd} 0 \arrow[r] & E \arrow[r] & \mc{E} \arrow[r] & \iota_*L \arrow[r] & 0\end{tikzcd}$$
	is a vector bundle with $c_1^2(E)=c_1^2(\mc{E})$ and $c_2(E)=c_2(\mc{E})-k$, where $\iota:\pi^{-1}(b)\to X$ is the inclusion map. Since $\Delta(E)\geq 0$, we must have $k\leq 2\Delta(\mc{E})$. We can iterate this process across all unstable fibres to see that there can only be finitely many, as all vector bundles have non-negative discriminant.
	
	We now consider the case where $\mc{E}$ is not locally free. In this case, $\mc{E}^{\vee\vee}/\mc{E}$ is a torsion sheaf supported at $m$ points with multiplicity, and $\mc{E}^{\vee\vee}$ is a vector bundle satisfying \begin{align*}
		\Delta(\mc{E}^{\vee\vee})=\Delta(\mc{E})-\frac{m}{r}, &&S_{\mc{E}^{\vee\vee}}=S_{\mc{E}}-\sum\limits_{x \in \text{Supp}(\mc{E}^{\vee\vee}/\mc{E})} \{\pi^{-1}(\pi(x))\}\times T^*.
	\end{align*}
	Since $\Delta(\mc{E}^{\vee\vee})\geq 0$, the support of $\mc{E}^{\vee\vee}/\mc{E}$ must be finite, allowing us to reduce to the first case. 
\end{proof}

By contrast, elementary modifications by positive-degree line bundles are highly non-unique; if $E|_{\pi^{-b}}\cong L\otimes (L^*\otimes \det(E)|_{\pi^{-1}(b)})$ with $L \in \Pic^h(T^*)$, $h\geq 0$ and $L\not\cong L^*\otimes \det(E)|_{\pi^{-1}(b)}$, then $E$ has an elementary modification at $b$ by $\lambda$ for every $\lambda \in \Pic^r(T^*)$ with $r\geq h$ \cite[Section 4.1.3]{MorHopf}. Two elementary modifications by $\lambda$ corresponding to maps $g_1:E|_{\pi^{-1}(b)}\to \lambda, g_2:E|_{\pi^{-1}(b)}\to \lambda$ are isomorphic if and only if there is a bundle automorphism $\varphi$ of $E$ so that $g_1=g_2\circ \varphi|_{\pi^{-1}(b)}$.

Finally, for the case of an elementary modification by a degree zero line bundle, the behaviour of the elementary modification depends on whether the initial bundle is \emph{regular}.
\begin{definition}
	A rank-2 vector bundle $E$ on $X$ is \emph{regular at $b$} for some $b \in B$ if $E|_{\pi^{-1}(b)}$ is semi-stable and not isomorphic to $\lambda \oplus \lambda$ for any $\lambda \in \Pic(T)$. Moreover, $E$ is \emph{regular} if it is regular at $b$ for every $b \in B$.
\end{definition}
The regular rank-2 vector bundles with fixed irreducible spectral curve $\overline{C}$ are classified via the following result:
\begin{proposition}[Br\^inz\u anescu--Moraru \cite{BrMoFM}]\label{prym}
	Let $\overline{C}$ be an irreducible bisection of $J(X)$ with normalization $C$. If we set $W$ to be the normalization of $X\times_B C$ with induced projections $\gamma:W\to X$ and $\psi:W\to C$, then 
	\begin{enumerate}
		\item[i.] There is a line bundle $L$ on $W$ such that $\gamma_*(L)$ is a regular rank-2 bundle on $X$ with spectral curve $\overline{C}$.
		\item[ii.] If $L_1,L_2\in \Pic(W)$ are such that $\gamma_*(L_1)$ and $\gamma_*(L_2)$ both have spectral curve $\overline{C}$, then $L_1\otimes L_2^{-1}=\psi^*(\lambda)$ for some $\lambda \in \Pic(C)$, and $\gamma_*(L_1)\cong \gamma_*(L_2)$ if and only if $L_1\cong L_2$.
	\end{enumerate} 
\end{proposition}
Using the notation of the above Proposition, we also have that if $L\in \Pic(W)$ is such that $\gamma_*(L)$ has spectral curve $\overline{C}$ and $\det(\gamma_*(L))=\delta$, then for any $\lambda \in \Pic(C)$ $$\det(\gamma_*(L\otimes \psi^*(\lambda))))\cong \delta \otimes \eta_n(\lambda),$$
where $\eta:C\to B$ is the projection induced from $J(X)$, and $\eta_n:\Pic(C)\to \Pic(B)$ is the norm homomorphism of $\eta$, which is the group homomorphism defined by $$\eta_n(\mc{O}_C(p))=\mc{O}_B(\eta(p))$$
for all $p \in C$. Because of this, the regular rank-2 bundles on $X$ with determinant $\delta$ and spectral curve $\overline{C}$ are of the form $\gamma_*(L\otimes \psi^*(\lambda))$, where $\lambda \in \text{Prym}(C/B):=\ker(\eta_n)$. (See \cite[Theorem 4.5]{BrMoFM} for more details.)

An elementary modification of a vector bundle $E$ at $b$ by a degree zero bundle $\lambda$ exists if and only if $E|_{\pi^{-1}(b)}$ is an extension of $\lambda$ by another degree-zero line bundle $\lambda'$. If $E$ is regular at $b$ there is a unique surjection from $E|_{\pi^{-1}(b)}$ to $\lambda$ up to composing with an automorphism of $\lambda$, so there is a unique elementary modification. If $E$ admits an elementary modification by $\lambda$ at $b$ but $E$ is not regular at b, then $E|_{\pi^{-1}(b)}\cong \lambda\oplus \lambda$ The surjections from $E|_{\pi^{-1}(b)}$ to $\lambda$ are parameterized by $\C^2$, and since constant multiples of a surjection induce the same elementary modification, the elementary modifications of $E$ by $\lambda$ at $b$ are parameterized by $\mb{P}^1$. 
\begin{lemma}\label{regular}
	Let $E$ be a rank-2 vector bundle with spectral curve $\overline{C}$, and let $E'$ be an elementary modification
	\begin{equation}\label{reg}\begin{tikzcd} 0 \arrow[r] & E' \arrow[r] & E \arrow[r] & \iota_*\lambda \arrow[r] & 0,
	\end{tikzcd}\end{equation}
	where $\iota:\pi^{-1}(b) \to X$ is the inclusion map for some $b \in B$ and $\lambda \in \Pic^0(T)$. Then $E'$ is regular at $b$ if and only if $E$ is.
\end{lemma}
\begin{proof}
	If $b$ is of the form $\pi(c)$ where $c$ is a smooth point of $\overline{C}$, then every vector bundle with spectral curve $\overline{C}$ is regular \cite{MorHopf}, so we can restrict to the case where $\pi^{-1}(b)$ contains a singular point of $\overline{C}$. In this case, a vector bundle $V$ with spectral curve $\overline{C}$ has that $V|_{\pi^{-1}(b)}$ is an extension of $\lambda$ by itself, and  \begin{equation}\label{h1} V \text{ is regular at } b \Leftrightarrow h^1(T,V|_{\pi^{-1}(b)}\otimes \lambda^{-1})=1.\end{equation}  Let $L \in \Pic^0(X)$  be a bundle with constant factor of automorphy so that $L|_{\pi^{-1}(b)}=\lambda^{-1}$. Then taking the tensor product of the exact sequence \eqref{reg} with $L$ gives $$\begin{tikzcd}
		0\arrow[r] & E'\otimes L \arrow[r] & E\otimes L \arrow[r] & \iota_*\mc{O}_T\arrow[r] & 0,
	\end{tikzcd}$$
	and pushing forward by $\pi$ gives the long exact sequence
	$$\begin{tikzcd}
		0\arrow[r] & \pi_*(E'\otimes L) \arrow[r] & \pi_*(E\otimes L) \arrow[r] & (\pi\circ \iota)_*\mc{O}_T \arrow[lld]\\
		&R^1\pi_*(E'\otimes L) \arrow[r] & R^1\pi_*(E\otimes L) \arrow[r] & R^1(\pi\circ \iota)_*\mc{O}_T \arrow[r] & 0.
	\end{tikzcd}$$
	If $V$ is any vector bundle on $X$ so that $H^0(T,V|_{\pi^{-1}(p)})\neq 0$ for at most finitely many $p \in B$, then $\pi_*V$ is zero, and $R^1\pi_*V$ is a torsion sheaf with stalks $(R^1\pi_*V)_p\cong H^1(T,V|_{\pi^{-1}(p})$. (For details see \cite{BrMoFM}.) Furthermore, it is easy to check that since $\pi\circ \iota:\pi^{-1}(b)\to B$ is a constant map, $(\pi\circ \iota)_*\mc{O}_T\cong R^1(\pi\circ\iota)_*\mc{O}_T\cong \C_b$, where $\C_b$ is the skyscraper sheaf supported at $b$. From this, the above exact sequence reduces to $$\begin{tikzcd}
		0\arrow[r] & \C_b\arrow[r] & R^1\pi_*(E'\otimes L)\arrow[r] & R^1\pi_*(E\otimes L)\arrow[r] & \C_b \arrow[r] & 0.
	\end{tikzcd}$$
	If we now consider the exact sequence induced from the stalks at $b$, we have $$\begin{tikzcd}0 \arrow[r] & \C\arrow[r] & H^1(T,E'|_{\pi^{-1}(b)}\otimes \lambda^{-1}) \arrow[lld]\\ H^1(T,E|_{\pi^{-1}(b)}\otimes \lambda^{-1})\arrow[r] & \C\arrow[r] & 0,
	\end{tikzcd}$$
	so $h^1(T,E'|_{\pi^{-1}(b)}\otimes \lambda^{-1})= h^1(T,E|_{\pi^{-1}(b)}\otimes \lambda^{-1})$. Together with \eqref{h1}, this implies that $E'$ is regular at $b$ if and only if $E$ is. 
\end{proof}
Pairing this lemma with Proposition \ref{prym} leads to the following result:
\begin{proposition}\label{irred}
	If $\overline{C}$ is an irreducible bisection of $J(X)$, then every vector bundle with spectral curve $\overline{C}$ is regular.
\end{proposition}
\begin{proof}
	Suppose that $E$ is a bundle with spectral curve $\overline{C}$. By \cite[Theorem 4.1]{BrMoFM}, there is a vector bundle $E_0$ with spectral curve $\overline{C}$ which is regular. $E_0$ is an elementary modification of the pushforward a line bundle on $W$ by a degree-zero bundle, so by Lemma \ref{regular} we can assume that $E_0=\gamma_*(L_0)$ for some $L_0 \in \Pic(W)$. We also have that there is some $L_1 \in \Pic(W)$ so that $E$ is an elementary modification of $\gamma_*(L_1)$ by a degree-zero bundle. Since $\gamma_*(L_0)$ and $\gamma_*(L_1)$ have the same spectral curve, there is a line bundle $V \in \Pic(C)$ so that $L_0\otimes L_1^{-1}\cong \psi^{*}(V)$, where $\psi:W\to C$ is the map induced from the fibred product \cite[Theorem 4.5]{BrMoFM}. Choose effective divisors $D_0,D_1$ on $C$ so that $V=\mc{O}_C(D_0-D_1)$. Since for any $c \in C$, the pushforward $\gamma_*(L\otimes \psi^*(\mc{O}_C(-c)))$ is an elementary modification of $\gamma_*(L)$ by a degree-zero bundle, there are sequences of elementary modifications by degree-zero bundles taking $E_0$ to $\gamma_*(L_1\otimes \psi^*(\mc{O}_C(-D_1)))$ and $E_0$ to $\gamma_*(L_0\otimes \psi^*(\mc{O}_C(-D_0)))$. Since $E_0$ is regular, we have that $\gamma_*(L_0\otimes \psi^*\mc{O}_C(-D_0)))\cong \gamma_*(L_1\otimes \psi^*(\mc{O}_C(-D_1)))$ is regular by Lemma \ref{regular}, and we similarly get that $E$ is regular since there is a chain of elementary modifications taking $\gamma_*(L_1\otimes \psi^*(\mc{O}_C(-D_1)))$ to $E$. 
\end{proof}

Since all of the rank-2 bundles with irreducible spectral curve $\overline{C}$ can be expressed as the pushforward of a line bundle on the normalization of $X\times_B C$, the bundles with determinant $\delta$ and spectral curve $\overline{C}$ are parameterized by $\text{Prym}(C/B)$ \cite[Theorem 4.5]{BrMoFM}. 

In the case of a rank-2 vector bundle $E$ with smooth and irreducible spectral curve $C$, we have $g(C)=4\Delta(E)+1$ \cite[Lemma 3.10]{BrMor}. Similar computations  give \begin{equation}p_a(S_\mc{E})=4\Delta(\mc{E})+1\label{genus}\end{equation}
for any rank-2 stably irreducible sheaf $\mc{E}$.
\begin{proposition}
	If $E$ is a rank-2 stably irreducible sheaf, and the spectral curve $S_E$ contains no jumps, then $S_E$ is smooth.
\end{proposition}
\begin{remark}
	For the case of $\Delta(2,\delta,c_2)=\frac{1}{4}$, this result is given in \cite[Corollary 4.4]{ApMorTom}.
\end{remark}
\begin{proof}
	Let $E$ be a rank-2 stably irreducible sheaf in $\mc{M}_{2,\delta,c_2}(X)$, and suppose for a contradiction that $S_E$ consists of a singular irreducible bisection with no jumps. Let $C$ be the normalization of $S_E$. Then $g(C)\leq 4\Delta(E)$ by \eqref{genus}, so $\dim(\mc{G}^{-1}(S_E))=\dim\text{Prym}(C/B)\leq4\Delta(E)-1$. Now take $E'$ to be any sheaf in $\mc{M}_{2,\delta,c_2}(X)$ whose spectral curve is smooth. Since the arithmetic genus of a spectral curve depends only on $\Delta(2,c_1(E),c_2(E))$, we have $g(S_{E'})=\rho_a(S_E)=4\Delta(E)+1$ from \cite[Lemma 3.10]{BrMor}. We assumed $S_{E'}$ is smooth, so we have $\dim(\mc{G}^{-1}(S_{E'}))=\dim\text{Prym}(S_{E'}/B)=4\Delta(E)$. Since a general stably irreducible sheaf has smooth spectral curve, the fibres of the graph map $\mc{G}:\mc{M}_{2,\delta,c_2}\to \mb{P}_{\delta,c_2}$ have dimension $4\Delta(E)$ outside a proper Zariski-closed subset of $\mb{P}_{\delta,c_2}$. This is a contradiction since the fibre dimension of a holomorphic map is upper semi-continuous, so $S_E$ is smooth when it has no jumps. 
\end{proof}

In the next section, we compute the fibres of the graph map in order to study the structure of $\mc{M}_{2,\delta,c_2}(X)$.

\section{Fibres of the Graph Map}\label{Fibres}
In this section, we attempt to parameterize the space $\mc{G}^{-1}(G)$ of torsion-free sheaves on $X$ with determinant $\delta$ and spectral curve $G$, where $G$ is of the form $\overline{C}+\sum\limits_{i=1}^k \{b_i\} \times T^*$ with $b_i \in B$ for each $i$, and where $\overline{C}$ is an irreducible bisection of $J(X)$ with normalization $C$. 

In the case where $G=\overline{C}$, every vector bundle  $E$ with spectral curve $\overline{C}$ can be described as an elementary modification \begin{equation}\begin{tikzcd}
		0 \arrow[r] & E \arrow[r] & \gamma_*L \arrow[r] & \iota_*\lambda \arrow[r] & 0
	\end{tikzcd}\label{deg0}\end{equation}
for some $b \in B$, $\lambda\in \Pic^0(\pi^{-1}(b))$, $L \in \Pic(W)$, where $\iota:\pi^{-1}(b)\to X$ is the inclusion, $W$ is the normalisation of $X\times_B C$, and $\gamma:W\to X$ the map coming from the fibred product \cite{ApTom}. Since $\gamma_*L|_{b}$ is a regular bundle by Proposition \ref{irred}, every elementary modification of the form \eqref{deg0} is itself the pushforward of a line bundle on $W$ \cite[Remarque 5]{ApTom}. We then have that $\mc{G}^{-1}(\overline{C})\cong \text{Prym}(C/B)$ \cite[Theorem 4.5]{BrMoFM}. Note that each connected component of $\text{Prym}(C/B)$ has dimension $g(C)-1$, and when $C$ is an unramified cover of degree $2$, $\text{Prym}(C/B)$ is the two-element group.

For the next simplest case, where $G=\overline{C}+\{b\}\times T^*$ and $\overline{C}\cong C$, we can compute the fibres of the graph map as in the following propositions:
\begin{proposition}\label{singular}
	Take $\delta \in \Pic(X)$, $b \in B$, and let $C\subset J(X)$ be a smooth and irreducible $\iota_\delta$-invariant bisection. Then the torsion-free sheaves with determinant $\delta$ and spectral curve $C+\{b\}\times T^*$ are parameterized by a union of two $\mb{P}^1$-bundles over $\text{Prym}(C/B)\times T$ that intersect along $|C\cap \{b\}\times T^*|$ sections.
\end{proposition}
\begin{remark}
	Since $C$ is a bisection of $\text{pr}_1:B\times T^*\to B$, $C\cap \{b\}\times T^*$ will always contain one or two points.
\end{remark}
\begin{proof}
	The fibre of the graph map above a spectral curve of this type can be decomposed into two irreducible components, with one component containing the non-locally free sheaves, and the other component containing the vector bundles. If $E$ is a vector bundle with determinant $\delta$ and spectral curve $C+\{b\}\times T^*$, then $E|_{\pi^{-1}(b)}$ is a degree zero vector bundle so that $h^1(T,E|_{\pi^{-1}(b)}\otimes \lambda)=1$ for every $\lambda \in T^*$ by Proposition \ref{FiniteUnstable}, so there is a line bundle $L \in \Pic^1(T)$ such that $E|_{\pi^{-1}(b)}\cong L\oplus \delta\otimes L^{-1}$. There is a unique choice of a vector bundle $\tilde{E}$ so that $\det(\tilde{E})=\delta\otimes \pi^* \mc{O}_B(-b)$, $c_2(\tilde{E})=c_2(E)-1$, and
	$$\begin{tikzcd}
		0 \arrow[r] &\tilde{E} \arrow[r] & E \arrow[r] & j_*(\delta_b\otimes L^{-1})\arrow[r] &0
	\end{tikzcd}$$
	is an exact sequence, where $j:\pi^{-1}(b)\to X$ is the inclusion map. In particular, $\tilde{E}$ has spectral curve $C$. We then have a well-defined projection $E\mapsto (\tilde{E}, L)\in \mc{G}^{-1}(C)\times \Pic^1(T)=\text{Prym}(C/B)\times T$. To determine the fibres of this map, we notice that for any choice of bundle $\tilde{E}$ with determinant $\delta\otimes \pi^*\mc{O}_B(b)$ and spectral curve $C$, $L \in \Pic^1(T)$, and  surjective map $f:\tilde{E}\to j_*L$, there is a vector bundle $E$ with determinant $\delta$ and spectral curve $C+\{b\}\times T^*$ 
	such that $E\otimes \pi^*\mc{O}_B(-b)=\ker(f)$. 
	
	Since $C$ is smooth and irreducible, either $$\tilde{E}|_{\pi^{-1}(b)}\cong \lambda\oplus (\delta|_{\pi^{-1}(b)}\otimes \lambda^{-1})$$ for some $\lambda \in \Pic^0(T)$ with $\lambda^{\otimes 2}\neq \delta|_{\pi^{-1}(b)}$, or $$\tilde{E}|_{\pi^{-1}(b)}\cong \lambda_0\otimes A,$$ where $A$ is the unique extension of $\mc{O}_T$ by $\mc{O}_T$ and $\lambda_0 \in \Pic^0(T)$ such that $\lambda_0^{\otimes 2}=\delta|_{\pi^{-1}(b)}$. In both cases, $\text{Hom}(\tilde{E}, j_*L)\cong \C^2$, and the non-surjective maps correspond to $|\{b\}\times T^* \cap C|$ 1-dimensional subspaces. Thus the vector bundles with determinant $\delta$ and spectral curve $C+\{b\}\times T^*$ are parameterized by a fibre bundle with base $\text{Prym}(C/B)\times T$ and fibre $\C$ or $\C^*$.
	
	If $\mc{E}$ is a non-locally free sheaf with determinant $\delta$ and spectral curve $C+\{b\}\times T^*$, then since the double dual of a torsion-free sheaf on a surface is locally free, $\mc{E}^{\vee\vee}$ is a vector bundle with determinant $\delta$ and spectral curve $C$. Since $\mc{E}$ has exactly one singularity, $\mc{E}^{\vee\vee}/\mc{E}$ is a torsion sheaf supported at a point $x \in \pi^{-1}(b)$. This gives a well-defined projection $\mc{E}\mapsto (\mc{E}^{\vee\vee}, \text{supp}(\mc{E}^{\vee\vee}/\mc{E})) \in \text{Prym}(C/B)\times T$. Since for any vector bundle $E$ with spectral curve $C$ and determinant $\delta$ the sheaves $\mc{E}$ with $\mc{E}^{\vee\vee}=E$ which are singular at a point are parameterized by $\text{Quot}(E,1)=\mb{P}(E)$ with the projection map sending $\mc{E}$ to $\text{supp}(E/\mc{E})$ \cite{EilLehn}, the non-locally free sheaves with determinant $\delta$ and spectral curve $\{b\}\times T^*+C$ are a $\mb{P}^1$-bundle with base $\text{Prym}(C/B)\times T$.
	
	As the union of these two components is the fibre of a holomorphic map between two compact spaces, the union of the components must be compact, and therefore the closure of the locally free component intersects with the non-locally free component along $|C\cap \{b\}\times T^*|$ sections. 
\end{proof}

In the previous proposition, we showed that for a spectral curve with a single jump, the fibre of the graph map decomposes into two irreducible components. The following result allows us to explicitly compute the intersection locus of these two components.

\begin{proposition}
	In the context of Proposition \ref{singular}, the intersection of the irreducible components of the fibre consists of sheaves of the form $\ker(f\oplus g)$ for $E$ a vector bundle with spectral curve $C$ and determinant $\delta\otimes \pi^*\mc{O}_B(b)$, $f:E\to j_*\lambda$ non-zero such that $(b,\lambda) \in C$,
	and $g: E\to \mc{O}_x$ with $x \in \pi^{-1}(b)$ such that $g\circ \ker(f)\neq 0$. Furthermore, these sheaves are determined up to isomorphism by $(E,\lambda,x)$.
\end{proposition}
\begin{proof}
	Since the non-locally free component in Proposition \ref{singular} is compact, we can compute the intersection by finding the limits of families of deformations within the locally free component which are not locally free. Every vector bundle in the locally free component is given by an elementary modification $$\begin{tikzcd}
		0 \arrow[r] & \tilde{E} \arrow[r] & E \arrow[r] & j_*L \arrow[r]  & 0
	\end{tikzcd},$$
	where $E$ is a vector bundle with spectral curve $C$ and determinant $\delta\otimes \pi^*\mc{O}_B(b)$, and $L \in \Pic^1(B)$. Fix a choice of $E$ and $L$. For any $\lambda$ such that $(b,\lambda) \in C$, there is a surjective sheaf map $\alpha:E\to j_*\lambda$, which is unique up to multiplication by a scalar. Let $p$ be the unique point in $T$ such that $L$ is an extension in $\Ext^1(\mc{O}_p, \lambda)$, and set $x=j(p)$. If we now choose sheaf maps $h:E\to j_*L$ and $\beta:E\to \mc{O}_{x}$ so that $h$ is surjective and $\beta\circ \ker(\alpha)\neq 0$, we now have that for any map $\eta:L\to \mc{O}_p$ that there is a unique $t \in \C$ satisfying $t(\eta\circ h)=\beta$. Using these data, we can construct a deformation over $\Ext^1(\mc{O}_p, \lambda)$ whose non-zero elements are given as follows:
	
	Given an extension $$\begin{tikzcd}0 \arrow[r] & \lambda \arrow[r, "\vp"] & L \arrow[r, "\psi"] & \mc{O}_p \arrow[r] &0\end{tikzcd}$$ corresponding to a non-zero $s  \in \Ext^1(\mc{O}_p, \lambda)$, define the maps $f_s=\vp\circ \alpha$ and $g_s=th$, where $t$ is chosen so $t(\psi\circ h)=\beta$, and set $\tilde{E}_s=\ker(f_s+g_s).$ Note that any other extension corresponding to $s$ will be given by maps $z\vp$ and $\frac{1}{z}\psi$ for some $z \in \C^*$, giving $f_s=z\vp\circ \alpha$ and $g_s=tzh$, so $(f_s+g_s)$ is unique up to multiplication by a scalar, and $\tilde{E}_s$ is well-defined. As $s$ goes to zero, these maps become $f_0=\iota_1\circ \alpha$ and $g_0=\iota_2\circ \beta$, where $\iota_i$ are the co-product maps for $j_*\lambda\oplus \mc{O}_{x}$. This gives that $\tilde{E}_0$ is of the desired form. 
	
	We now show that $\tilde{E}_0$ is independent of the choice of $\beta$. Clearly, $\ker(f_0+g_0)=\ker(\alpha)\cap \ker(\beta)$. Since $\beta$ vanishes away from $x$, $\tilde{E}_0$ may only depend on $\beta$ in the fibre above $x$. The linear maps $\alpha|_{x}$ and $\beta|_{x}$ both have rank one, so either $\ker(\alpha|_{x})=\ker(\beta|_{x})$ or $\ker(\alpha|_{x})\cap \ker(\beta|_{x})=0$. But $\ker(\alpha|_{x})\neq \ker(\beta|_{x})$, since $\beta\circ \ker\alpha\neq 0$, meaning that all the sections of $\tilde{E}_0$ vanish at $x$. Thus $\tilde{E}_0$ does not depend on $\beta$. 
\end{proof}

\section{Applications}\label{Applications}
In this section, we use Proposition \ref{spectral} and the results of section 4 to prove some results about the fundamental groups of moduli spaces of rank-2 stably irreducible sheaves, as well as compute explicit data about the graph map fibration in the cases where the dimension of the moduli space is at most 6. For this section, the invariants $\delta,c_2$ are assumed to be such that a rank-2 sheaf $\mc{E}$ is stably irreducible whenever $\det(\mc{E})=\delta$ and $c_2(\mc{E})=c_2$.

\subsection{The Topology of the Moduli Spaces}
The case of $\Delta(2,\delta,c_2)=\frac{1}{4}$ was previously studied in \cite{ApMorTom} leading to the following result:
\begin{proposition}[Aprodu--Moraru--Toma]\label{2dim}
	Let $\delta \in \Pic(X)$ and $c_2 \in \Z$ be such that $\mc{M}_{2,\delta,c_2}(X)$ is 2-dimensional and $t(2,\delta)>\frac{1}{4}$. Then $\mc{M}_{2,\delta,c_2}(X)$ is a primary Kodaira surface with the same fibre as $X$, and their N\'eron--Severi groups satisfy the relation $$\ord(NS(X))|\ord(NS(\mc{M}_{2,\delta,c_2}(X))).$$
\end{proposition}
\begin{proposition}
	Let $\delta,c_2$ be such that $\Delta(2,\delta,c_2)< t(2,\delta)$.
	\begin{enumerate}
		\item[i.] If $\Delta(2,\delta,c_2)=0,$ then $\mc{M}_{2,\delta,c_2}(X)$ consists of four points. 
		\item[ii.] If $\Delta(2,\delta,c_2)>0$, then the induced map of fundamental groups $$\pi_1(\mc{G}):\pi_1(\mc{M}_{2,\delta,c_2}(X))\to \pi_1(\mb{P}_{\delta,c_2})\cong \Z^2$$ is surjective. 
	\end{enumerate}
	In particular, $\mc{M}_{2,\delta,c_2}(X)$ is not simply connected when $\Delta(2,\delta,c_2)>0$. 
\end{proposition}
\begin{proof}
	For the case of $\Delta(2,\delta,c_2)=0$, every spectral curve is smooth by Proposition \ref{FiniteUnstable}, and the genus formula \eqref{genus} gives that each spectral curve $C$ in $\mb{P}_{2,\delta,c_2}$ is an unramified double cover of $B$. From this we can conclude that $\mc{G}^{-1}(C)\cong \text{Prym}(C/B)$ is the group with two elements, giving the desired result.
	
	The case of $\Delta(2,\delta,c_2)=\frac{1}{4}$ is a direct corollary of Proposition \ref{2dim}, so for the remainder of the proof we assume $\Delta(2,\delta,c_2)\geq \frac{1}{2}$.
	
	Recall that given any fibre bundle $F\hookrightarrow Y\to Z$, there is an induced exact sequence 
	\begin{equation}\begin{tikzcd}\label{homotopy}\pi_1(F)\to \pi_1(Y)\to \pi_1(Z)\to \pi_0(F) \end{tikzcd} \end{equation}
	of the homotopy groups \cite[Section 17]{BottTu}.
	
	Using \eqref{homotopy} we see that whenever $(\delta,c_2)$ are such that $\frac{1}{2}\leq \Delta(2,\delta,c_2)<t(2,\delta)$,we have $$\pi_1(\mb{P}_{\delta,c_2})=\pi_1(B)=\Z^2$$ since $\mb{P}_{\delta,c_2}$ is a holomorphic fibre bundle with connected and simply connected fibres. In particular this means that for any section $\sigma:B\to \mb{P}_{\delta,c_2}$ and any element $[\gamma]\in \pi_1(\mb{P}_{\delta,c_2})$, there is a representative of $[\gamma]$ contained in $\sigma(B)$. Let $E$ be a regular rank-2 vector bundle in $\mc{M}_{2,\delta,c_2-1}(X)$ with spectral curve $C$, and take the section $\sigma_E:B\to \mb{P}_{\delta,c_2}$ given by $b\mapsto C+\{b\}\times T^*$. We will show that for any loop $\gamma\in \sigma_E(B)$, there is a loop in $\mc{M}_{2,\delta,c_2}(X)$ which maps to $\gamma$, demonstrating that $\pi_1(\mc{G})$ is a surjection.
	Consider the subset $\text{Quot}(E,1)\subseteq \mc{M}_{2,\delta,c_2}$ consisting of non-locally free sheaves whose double dual is $E$ and which have one singularity counting multiplicity. Since $\text{Quot}(E,1)\cong \mb{P}(E)$, the map $\mc{G}|_{\text{Quot}(E,1)}$ is a fibration over $\sigma_E(B)$. Applying \eqref{homotopy} again gives the exact sequence $$\begin{tikzcd}
		\pi_1(\mb{P}(E|_{\pi^{-1}(b)}))\arrow[r] & \pi_1(\text{Quot}(E,1))\arrow[r] & \pi_1(\sigma_E(B))\arrow[r] & 0.
	\end{tikzcd}$$ 
	Since every element $[\gamma]\in \pi_1(\mb{P}_{\delta,c_2})$ can be represented by a loop in $\sigma_E(B)$, factoring through inclusion into $\mc{M}_{2,\delta,c_2}(X)$ gives that the map $$\pi_1(\mc{G}):\pi_1(\mc{M}_{2,\delta,c_2}(X))\to \pi_1(\mb{P}_{\delta,c_2})$$ is surjective. 
\end{proof}
\begin{remark}
    By using \cite[Lemma 3.5]{Arapura} we can improve the above result to get an exact sequence $$\begin{tikzcd}
    \Z^{8\Delta(2,\delta,c_2)} \arrow[r] & \pi_1(\mc{M}_{2,\delta,c_2}(X)) \arrow[r] & \pi_1(\mb{P}_{2,\delta}) \arrow[r] & 0
    \end{tikzcd}$$
    when $\Delta(2,\delta,c_2)\leq \frac{3}{4}$. This result may also work in higher dimensions, but has not yet been verified.
\end{remark}
\begin{proposition}
	Let $\delta\in \Pic(X)$ and $c_2\in \mathbb{Z}$ be such that $0<\Delta(2,\delta,c_2)<t(2,\delta)$. Then $\mc{M}_{2,\delta,c_2}(X)$ does not admit a K\"ahler structure.
\end{proposition}
\begin{proof}
	The case of $\Delta(2,\delta,c_2)=\frac{1}{4}$ follows immediately from Proposition \ref{2dim}. For $\Delta(2,\delta,c_2)\geq \frac{1}{2}$, let $E$ be a regular vector bundle in $\mc{M}_{2,\delta,c_2-1}(X).$ Then $\text{Quot}(E,1)\cong \mb{P}(E)$ is a complex submanifold of $\mc{M}_{2,\delta,c_2}(X)$. Since $\mb{P}(E)$ is a fibre bundle with simply connected fibre, $b_1(\mb{P}(E))\cong b_1(X)=3$, so $\mb{P}(E)$ does not admit any K\"ahler structure. Since any complex submanifold of a K\"ahler manifold has an induced K\"ahler structure, $\mc{M}_{2,\delta,c_2}$ can not admit a K\"ahler structure. 
\end{proof}
\begin{corollary}
	When $0<\Delta(2,\delta,c_2)<t(2,\delta)$, the moduli space $\mc{M}_{2,\delta,c_2}(X)$ is a non-K\"ahler compact holomorphic symplectic manifold which is not deformation equivalent to a Bogomolov-Guan manifold.
\end{corollary}
In particular, the above result implies that either $\mc{M}_{2,\delta,c_2}(X)$ is deformation equivalent to a Douady space of points on a Kodaira surface, or $\mc{M}_{2,\delta,c_2}(X)$ belongs to a deformation class separate from other known examples of compact holomorphic symplectic manifolds.

\subsection{4- and 6-Dimensional Moduli Spaces}
In the case where $\Delta(2,\delta,c_2)\in \{\frac{1}{2},\frac{3}{4}\}$, Proposition \ref{FiniteUnstable} gives that spectral curves in $\mb{P}_{\delta,c_2}$ contain at most one jump counting multiplicity. In these cases, we can describe the fibres of the graph map above all spectral curves using Proposition \ref{singular} and \cite[Theorem 4.5]{BrMoFM}.

\begin{proposition}\label{4dim}
	In the case that $\Delta(2,\delta,c_2)=\frac{1}{2}$ and $t(2,\delta)>\frac{1}{2}$, the space of graphs $\mb{P}_{\delta,c_2}$ is a ruled surface with base $B$, and $\mb{P}_{\delta,c_2}^1$ is a bi-section of $\mb{P}_{\delta,c_2}\to B$. The fibres of $\mc{G}$ above points in $\mb{P}^0_{\delta,c_2}$ are 2-dimensional Prym varieties, and the fibres above points in $\mb{P}_{\delta,c_2}^1$ are given by a union of two ruled surfaces with base $T\times \{1,-1\}$, which intersect along two sections. 
\end{proposition}
\begin{remark}
	The set $\{1,-1\}$ in the above statement corresponds to the Prym variety $\text{Prym}(C/B)$, where $C$ is a bisection of $J(X)$ such that $C+\{b\}\times T^*\in \mb{P}_{\delta,c_2}$ for some $b \in B$.
\end{remark}
\begin{proof}
	For these invariants, since $\Delta(2,\delta,c_2)=\frac{1}{2}>0$, Proposition \ref{spectral} gives that $\mb{P}_{\delta,c_2}$ is a ruled surface with base $B$, and as in Remark \ref{filter}, the graphs with jumps are parameterized by $\mb{P}^0_{\delta,c_2-1}\times \text{Hilb}^1(B)=\mb{P}_{\delta,c_2-1}\times B$. Since we have $\Delta(2,\delta,c_2-1)=0$, the space of graphs $\mb{P}_{\delta,c_2-1}$ is a two point set by Proposition \ref{spectral}.  Using the genus formula \eqref{genus}, we see that the spectral curves in this scenario can be either a genus 3 curve $C$ or a genus 1 curve $C'$ plus a jump of length one at some $b \in B$. This immediately gives that the fibres $\text{Prym}(C/B)$ above spectral curves $C$ are 2-dimensional. For the fibres above graphs with a jump, Proposition \ref{singular} gives that both the non-locally free component of the fibre and the closure of the locally free component are ruled surfaces with base $T\times \text{Prym}(C'/B)\cong T\times \{1,-1\}$, and they intersect along $|C'\cap \{b\}\times T^*|$ sections. As any map from a genus 1 curve to $B$ is an unramified covering map, $|C'\cap \{b\}\times T^*|=2$ for all $b \in B$. this implies that the locally free and non-locally free components will intersect along two sections of the ruled surfaces, regardless of the position of the jump. 
\end{proof}
\begin{remark}
Note that in the 4-dimensional case the singular fibres of the graph map are similar to those in the natural Lagrangian fibration on $X^{[2]}$ as discussed in section \ref{Douady}.    
\end{remark}

\begin{proposition}\label{6dim}
	In the case that $\Delta(2,\delta,c_2)=\frac{3}{4}$, the space of graphs $\mb{P}_{\delta,c_2}$ is a $\mb{P}^2$-bundle with base $B$, and $\mb{P}^1_{\delta,c_2}$ is isomorphic to $B\times B$. The fibres of $\mc{G}$ above points in $\mb{P}_{\delta,c_2}^0$ are 3-dimensional Prym varieties, and the fibres above spectral curves $C+\{b\}\times T^* \in \mb{P}^1_{\delta,c_2}$ are given by a union of two $\mb{P}^1$-bundles with base $T\times T$, which intersect along $|C\cap \{b\}\times T^*|$ sections.
\end{proposition}
\begin{proof}
	This case is analogous to Proposition \ref{4dim}. Proposition \ref{spectral} gives that $\mb{P}_{\delta,c_2}$ is a $\mb{P}^2$-bundle with base $B$, and the spectral curves with a jump are parameterized by $\mb{P}_{\delta,c_2-1}\times B$. Since $\Delta(2,\delta,c_2-1)=\frac{1}{4}$, we have $\mb{P}_{\delta, c_2-1}\cong B$. The genus formula \eqref{genus} gives that spectral curves can be either a genus 4 curve $C$, or a genus 2 curve $C'$ plus a jump of length one at some $b \in B$. When $C$ is genus 4, $\text{Prym}(C/B)$ is 3-dimensional. For the fibres above graphs with a jump, Proposition \ref{singular} gives that both the non-locally free component of the fibre and the closure of the locally free component are $\mb{P}^1$-bundles with base $T\times \text{Prym}(C'/B)\cong T\times T$, and they intersect along $|C'\cap \{b\}\times T^*|$ sections. 
\end{proof}
\begin{remark}
	For the case of Proposition \ref{6dim}, since $C'\to B$ is a degree 2 map from a genus 2 curve to a genus 1 curve, the map has ramification at two points. Thus $|C'\cap \{b\}\times T^*|=1$ if $b$ is the image of a ramification point of $C'\to B$, and $|C'\cap \{b\}\times T^*|=2$ otherwise.
\end{remark}

\begin{remark}
Note that in this case all of the singular fibres of the graph map have a similar complexity, as the singular fibres correspond after allowable elementary modifications to moduli spaces of sheaves of dimensions $6-4k$, for some $k\in \Z_{>0}$, of which $2$ is the only non-negative value. This contrasts with the Douady space $X^{[3]}$, where the fibres above points of the form $3p \in \text{Sym}^3(B)$ have significantly different behaviour to the singular fibres above points of the form $2p+q \in \text{Sym}^3(B)$. This can be seen by comparing punctual Hilbert schemes of 2 and 3 points as in \cite{Briancon}.
\end{remark}

\subsection{Higher Dimensions}
In general, when $\Delta(2,\delta,c_2)>0$, the generic spectral curve is a smooth curve $C$ of genus $4\Delta(2,\delta,c_2)+1$, and the fibre above $C$ is given by $\text{Prym}(C/B)$. For a spectral curve $S=C+\sum\limits_{i=1}^k \mu_i\{b_i\}\times T^*$, the non-locally free component of $\mc{G}^{-1}(S)$ can be found by describing $\text{Quot}(E,\ell)$ for all vector bundles $E$ in $\mc{M}_{2,\delta,c_2-\ell}(X)$ with spectral curve $S_E=S-\sum\limits_{i=1}^k \nu_i \{b_i\}\times T^*$ for $\{\nu_i\}$ with $0\leq \nu_i\leq \mu_i$ and $\sum\limits_{i=1}^k \nu_i=\ell$.

The vector bundles with spectral curve $S$ can be described by parameterizing the sequences of elementary modifications taking a vector bundle with spectral curve $C$ to one with spectral curve $S$. This process is described in detail in \cite[Section 4]{MorHopf} for Hopf surfaces, and the method for Kodaira surfaces is similar.

Because of this, the fibres of the graph map above spectral curves with jumps must be computed inductively using information about moduli spaces of lower dimensions, so an understanding of the fibration structure of $\mc{M}_{2,\delta,c_2}(X)$ requires a description of the fibration structure of $\mc{M}_{2,\delta,c_2-k}(X)$ for all $0\leq k\leq 2\Delta(2,\delta,c_2)$.

\subsection{Further Questions}
In the case of a Lagrangian fibration $f:M\to P$ with both $M$ and $P$ K\"ahler, a result from \cite{SoldVerb} (due to Matsushita \cite{Matsushita} in the projective case) gives an isomorphism between $R^i\pi_*\mc{O}_M$ and $\Omega^i_P$ for integers $i$, from which the cohomology of $\mc{O}_M$ can be computed from the Hodge numbers of $P$ via the Leray spectral sequence. The proof of the above result uses the K\"ahler condition mainly to show the isomorphism away from the singular fibres of the Lagrangian fibration, so if $\mc{M}_{2,\delta,c_2}(X)$ has a K\"ahler metric away from the singular fibres of its Lagrangian fibration the above result may still hold in this case. Under these hypotheses the Leray spectral sequence would degenerate at the second page, giving $$H^i(\mc{O}_{\mc{M}_{2,\delta,c_2}(X)})=\begin{cases}\C & i=0,\\
\C^2 & 0<i<8\Delta(2,\delta,c_2),\\
\C & i=8\Delta(2,\delta,c_2).
\end{cases}$$

With respect to the fundamental group, in addition to determining whether the bounds from \cite{Arapura} extend to higher dimensions, these bounds may also be improved by studying when loops in the smooth fibres are homotopy equivalent to loops in a singular fibre. This would naturally generalize the case of elliptic surfaces with singular fibres but no multiple fibres, where the fundamental group is entirely determined by the base as all loops in smooth fibres are homotopy equivalent to loops inside the simply connected singular fibres. Proving such a result for $\mc{M}_{2,\delta,c_2}(X)$ would improve the bounds on the number of generators of the fundamental group to $$\begin{tikzcd}\Z^{8\Delta(2,\delta,c_2)-2} \arrow[r] & \pi_1(\mc{M}_{2,\delta,c_2}(X)) \arrow[r] & \pi_1(\mb{P}_{2,\delta}) \arrow[r] & 0
\end{tikzcd}$$
for $\Delta(2,\delta,c_2)\geq \frac{1}{2}$. In the $\Delta(2,\delta,c_2)=\frac{1}{2}$ case, this would imply that the fundamental group of $\mc{M}_{2,\delta,c_2}$ has at most four generators, which is exactly the number of generators for the fundamental group of $X^{[2]}$.

Another avenue of future study for this problem is to perform further comparisons of the Lagrangian fibration of $\mc{M}_{2,\delta,c_2}$ with the natural Lagrangian fibration on $X^{[n]}$. While the Lagrangian fibrations for these two families are both $\mb{P}^n$-bundles over $B$, they are never isomorphic. At this point it is unclear if this difference in base for the Lagrangian fibrations is sufficient to force the $\mc{M}_{2,\delta,c_2}(X)$ to be distinct from Douady spaces or if it can be reconciled to get spaces which are still deformation equivalent, perhaps by composing with a finite cover. This contrasts with the case of an elliptically fibred abelian surface, where for appropriately chosen invariants ($\gcd(c_1\cdot f, r)=1$ with $f$ a fibre of the elliptic fibration) every stable bundle is uniquely determined by its allowable elementary modifications up to twisting by a line bundle \cite[Chapter 8, Proposition 9]{Fried}. This fact directly gives the correspondence between between moduli spaces of stable bundles and Hilbert schemes of points in this case. One way in which we can investigate this difference in the future is to look moduli spaces of stable sheaves on a product $E_1\times E_2$ of elliptic curves and its spectral construction with respect to both natural elliptic fibrations. If the rank $r$ and first Chern class $c_1$ are chosen so that $\gcd(r,c_1\cdot f_1)=1$ and $r | (c_1\cdot f_2)$ with $f_1$ and $f_2$ general fibres of the two fibrations, the spectral constructions will immediately give the Hilbert scheme structure for the first fibration and behaviour similar to the Kodaira case for the second fibration. This question will will be analysed in a future paper.

\addcontentsline{toc}{section}{References}
	\bibliographystyle{alpha}
	\bibliography{ref}
\end{document}